\documentclass[11pt, a4paper, twoside]{amsart}
\usepackage{a4, latexsym, amsfonts,epsf, amsmath, amsthm, amssymb, xy, stmaryrd, graphicx, tikz}
\usetikzlibrary{matrix,arrows}
\usepackage[utf8]{inputenc}
\usepackage{cyrillic}
\usepackage{calligra}
\usepackage{fancyhdr}
\xyoption{all}

\newtheoremstyle{swappedplain}{\topsep}{\topsep}%
{\itshape}
{}
{\bfseries}
{.}
{ }
{\thmnumber{(#2)}\thmname{ #1}\thmnote{ #3}}

\newtheoremstyle{swappeddefinition}{\topsep}{\topsep}%
{}
{}
{\bfseries}
{.}
{ }
{\thmnumber{(#2)}\thmname{ #1}\thmnote{ #3}}

\theoremstyle{definition}

\newtheorem{defi}{Definition}[section]
\newtheorem{rema}[defi]{Remark}

\newtheorem{exam}[defi]{Example}

\theoremstyle{plain}

\newtheorem{prop}[defi]{Proposition}

\newtheorem{theo}[defi]{Theorem}
\newtheorem{coro}[defi]{Corollary}
\newtheorem{lemm}[defi]{Lemma}
\newtheorem*{theono}{Theorem}

\let\copyref\ref
\renewcommand*{\ref}[1]{{\rm (\copyref{#1})}}

\newcommand{\f}[1]{\it #1\rm}							
\newcommand{\llangle}{\langle\!\langle}						
\newcommand{\rrangle}{\rangle\!\rangle}						
\newcommand{\cyrrm}{\fontencoding{OT2}\selectfont\textcyrmd} 			

\newcommand{\resprod}{\setbox0 = \hbox{$\displaystyle\prod$}			
    \mathop{\vtop{\copy0\kern -1.6pt \hrule}}}					

\newcommand{\prlim}[1]{\lim\limits_{\substack{\longleftarrow\\ #1}}}		
\newcommand{\cd}{\mbox{\it cd}\ }						
\newcommand{\infl}{\mbox{\it inf}\ }						
\newcommand{\inflv}{\mbox{\it inf}\ ^\vee}					
\newcommand{\tg}{\mbox{\it tg}}							
\newcommand{\tr}{\mbox{\it tr}}							
\newcommand{\Gal}{\operatorname{Gal}}						

\newcommand{\dimfp}{\dim_{\mathbb{F}_p}}
\newcommand{\rusb}{\textup{\cyrrm{B}}}						
\newcommand{\Sha}{\textup{\cyrrm{Sh}}}						
\newcommand{\modulo}{\operatorname{mod}}
\newcommand{\xyalign}{\entrymodifiers={+!!<0pt,\fontdimen22\textfont2>}}
\newcommand{\textxymatrix}[1]{\xymatrix@C=12pt{#1}}				
\newcommand{\mf}[1]{\mathfrak{#1}}						
\newcommand{\mc}[1]{\mathcal{#1}}						
\newcommand{\gr}{\operatorname{gr}}						
\newcommand{\grg}{\mathbb{F}_p\llbracket G \rrbracket}				
\newcommand{\power}{\mathbb{F}_p\llangle X \rrangle}				
\newcommand{\polyx}{\mathbb{F}_p\langle X\rangle}				
\newcommand{\redeiinvnotilde}{[l_j, l_k, l_i]}					
\newcommand{\stackrels}[2]{\stackrel{#1}{\mbox{\tiny $#2$}}} 			
\newcommand{\ocom}[3]{[[\ol{x}_{#1},\ol{x}_{#2}],\ol{x}_{#3}]}			
\newcommand{\Q}{\mathbb{Q}}							
\newcommand{\N}{\mathbb{N}}							
\newcommand{\Z}{\mathbb{Z}}							
\newcommand{\F}{\mathbb{F}}							
\newcommand{\R}{\mathbb{R}}							
\newcommand{\C}{\mathbb{C}}							
\newcommand{\ol}[1]{\overline{#1}}						

\newcommand{\CC}{\mc{C}}							
\newcommand{\z}{\mf{z}}								
\newcommand{\lbracket}{[}							
\newcommand{\rbracket}{]}							

\long\def\symbolfootnote[#1]#2{\begingroup%
\def\thefootnote{\fnsymbol{footnote}}\footnote[#1]{#2}\endgroup}

	\renewcommand{\thefootnote}{}

\begin{document}

\pagestyle{fancy}
\markleft{\it Gärtner, Rédei symbols and arithmetical mild pro-$2$-groups}
\markright{\it Gärtner, Rédei symbols and arithmetical mild pro-$2$-groups}

\begin{abstract}
Generalizing results of Morishita and Vogel, an explicit description of the triple Massey product for the Galois group $G_S(2)$ of the maximal $2$-extension of $\Q$ unramified outside a finite set of prime numbers $S$ containing 2 is given in terms of Rédei symbols. We show that mild pro-$2$-groups with Zassenhaus invariant $3$ occur as Galois groups of the form $G_S(2)$. Furthermore, a non-analytic mild fab pro-$2$-group having only $3$ generators is constructed .
\end{abstract}

\footnotetext{\it key words and phrases. \rm Mild pro-$p$-groups, Massey products, restricted ramification, fab pro-$p$-groups.}

\title{Rédei symbols and arithmetical mild pro-$2$-groups}
\author{Jochen Gärtner\\
\today}
\maketitle

\thispagestyle{empty}

\section{Introduction}

Galois extensions with ramification restricted to finite sets of primes arise naturally in algebraic number theory and arithmetic geometry, e.g.\ in terms of representations coming from the Galois action on the (étale) cohomology of algebraic varieties defined over number fields. If $k$ is a number field, $p$ a prime number and $S$ is a finite set of primes of $k$, we denote by $k_S(p)$ the maximal pro-$p$-extension of $k$ unramified outside $S$ and by $G_S(p) = \Gal(k_S(p)|k)$ its Galois group. If $S$ contains the archimedian primes and the set $S_p$ of primes of $k$ lying above $p$, the group $G_S(p)$ is fairly well understood. In particular it has been known that it is of cohomological dimension $\cd G_S(p) \le 2$ (assuming $k$ is totally imaginary if $p=2$) and often a duality group. In the \f{tame case} however, i.e.\ if $S\cap S_p=\varnothing$, the structure of these groups has remained rather mysterious for a long time. Some of the few known results are still conjectural such as the \f{Fontaine-Mazur conjecture} which predicts that these groups are either finite or not $p$-adic analytic.
\vspace{10pt}

In his fundamental paper \cite{JL}, J.\ Labute could give the first examples of Galois groups of the form $G_S(p)$ over $\Q$ in the tame case where $\cd G_S(p)=2$ using the theory of \f{mild pro-$p$-groups}. This has been the starting point for more far-reaching studies. We'd like to mention the following remarkable result due to A.\ Schmidt:

\begin{theono}[Schmidt]
 Let $p$ be an odd prime number, $k$ a number field and $S$ a finite set of primes of $k$. Let $\mc{M}$ be an arbitrary set of primes of $k$ with Dirichlet density $\delta(M)=0$. Then there exists a finite set $S_0$ disjoint from $S\cup \mc{M}$ such that $\cd G_{S\cup S_0}(p)=2$.
\end{theono}

For the precise (and even stronger) statement see \cite{AS2}, Th.1.1. Note that in the above theorem the set $S_0$ can always be chosen to be disjoint from $S_p$. 
\vspace{10pt}

As an important ingredient in the proof of the above theorem, using results of J.\ Labute, A.\ Schmidt showed that a finitely presented pro-$p$-group $G$ is mild (and hence of cohomological dimension $\cd G=2$) if the cohomology group $H^1(G,\F_p)$ admits a direct sum decomposition $H^1(G,\F_p)=U\oplus V$ such that the cup-product $\cup: H^1(G,\F_p)\otimes H^1(G,\F_p)\to H^2(G,\F_p)$ maps $U\otimes V$ surjectively onto $H^2(G,\F_p)$ and is identically zero on $V\otimes V$. Note that this cup-product being surjective is equivalent to $G$ having a minimal system of defining relations which are linearly independent modulo the third step of the \f{Zassenhaus filtration}. In \cite{JGMasseyMild}, the author generalizes Schmidt's cup-product criterion to finitely presented pro-$p$-groups with arbitrary \f{Zassenhaus invariant} using higher \f{Massey products}. However, it does not follow from Schmidt's result that mild pro-$p$-groups with Zassenhaus invariant $>2$ occur as Galois groups of the form $G_S(p)$. The goal of this paper is to give an affirmative answer to this question. More precisely, we will prove the following result in the case $k=\Q, p=2$:

\begin{theono}[Th.\ref{wildexample}] 
 Let $S=\{2,l_1,\ldots, l_n\}$ for some $n\ge 1$ and prime numbers $l_i\equiv 9\mod 16,\ i=1,\ldots, n$, such that the Legendre symbols satisfy
\[
 \left(\frac{l_i}{l_j}\right)_2=1,\ 1\le i,j\le n,\ i\neq j.
\]
Then $G_S(2)=\Gal(\Q_S(2)|\Q)$ is a mild pro-$2$-group with generator rank $n+1$, relation rank $n$ and Zassenhaus invariant $\z(G)=3$.
\end{theono}

We will show that the triple Massey product for the group $G_S(2)$ is amenable to an explicit description via \f{Rédei symbols}. Furthermore, we consider certain \f{fab} quotients of the groups $G_S(2)$ and give an explicit description of their triple Massey products, cf.\ Th.\ref{GSTpres}. This will enable us to construct an explicit example of a mild, non-analytic fab pro-$2$-group with Zassenhaus invariant $3$.
\vspace{10pt}

\bfseries Acknowledgements: \mdseries Parts of the results in this article are contained in the author's PhD thesis. The author likes to thank Kay Wingberg for his guidance and great support and to Alexander Schmidt and Denis Vogel for helpful suggestions. Further thanks go to Hugo Chapdelaine and Claude Levesque for providing a relevant reference concerning the $2$-class group of a quadratic number field.

\section{Review of Massey products in the cohomology of pro-$p$-groups}

In this section we recall the definition and properties of higher Massey products for group cohomology of pro-$p$-groups. We will not give any proofs and refer the reader to the existing literature (e.g.\ see \cite{DK} and \cite{DV2}).

Let $p$ be a prime number. With a view towards our applications, for a pro-$p$-group $G$ we consider the trivial $G$-module $\F_p$ only and set
\[
 H^i(G)=H^i(G,\F_p),\ h^i=\dimfp H^i(G).
\]

By $\CC^*(G)=\CC^*(G,\F_p)$ we denote the standard inhomogeneous cochain complex (e.g.\ see \cite{NSW}, Ch.I, §2).

\begin{defi}
 Let $n\ge 2$ and $\alpha_1,\ldots, \alpha_n\in H^1(G)$. We say that the \f{$n$-th Massey product $\langle \alpha_1,\ldots, \alpha_n\rangle_n$} is \f{defined} if there is a collection
\[
 \mc{A}=\{a_{ij}\in \CC^1(G)\ |\ 1\le i,j\le n,\ (i,j)\neq (1,n)\}
\]
(called a \f{defining system} for $\langle \alpha_1,\ldots, \alpha_n\rangle_n$), such that the following conditions hold:
\begin{itemize}
 \item[\rm (i)] $a_{ii}$ is a representative of the cohomology class $\alpha_i,\ 1\le i\le n$.
 \item[\rm (ii)] For $1\le i < j\le n,\ (i,j)\neq (1,n)$ it holds that
 \[
  \partial^2 (a_{ij}) = \sum_{l=i}^{j-1} a_{il}\cup a_{(l+1)j}
 \]
\end{itemize}
where $\partial^2$ denotes the coboundary operator $\partial^2: \CC^1(G)\longrightarrow \CC^2(G)$.
If $\mc{A}$ is a defining system for $\langle \alpha_1,\ldots, \alpha_n\rangle_n$, the $2$-cochain
\[
 b_\mc{A} = \sum_{l=1}^{n-1} a_{1l}\cup a_{(l+1)n}
\]
is a cocyle and we denote its class in $H^2(G)$ by $\langle \alpha_1,\ldots, \alpha_n\rangle_\mc{A}$. We set
\[
 \langle \alpha_1,\ldots, \alpha_n\rangle_n = \bigcup_{\mc{A}} \langle \alpha_1,\ldots, \alpha_n\rangle_\mc{A}
\]
where $\mc{A}$ runs over all defining systems. The Massey product $\langle \alpha_1,\ldots, \alpha_n\rangle_n$ is called \f{uniquely defined} if $\#\langle \alpha_1,\ldots, \alpha_n\rangle_n=1$. We say that the $n$-th Massey product is \f{uniquely defined for $G$} if $\langle \alpha_1,\ldots, \alpha_n\rangle_n$ is uniquely defined for all $\alpha_1,\ldots, \alpha_n\in H^1(G)$.
\end{defi}

The $2$-fold Massey product is uniquely defined given by the cup-product $\cup: H^1(G)\times H^1(G)\to H^2(G)$ . If this cup-product is identically zero, the triple Massey product is uniquely defined. More generally, if $n\ge 2$ and the $k$-th Massey product is uniquely defined and identically zero for all $k<n$, the $n$-th Massey product is also uniquely defined and yields a multilinear map of $\F_p$-vector spaces
\[
\langle \cdot, \ldots, \cdot \rangle_n: H^1(G)^n\longrightarrow H^2(G).
\]

\begin{rema}
 The $n$-th Massey products commute with inflation, restriction and corestriction homomorphisms provided they are uniquely defined. In fact, this follows directly from their definition and the functoriality properties of the cup-product on the level of cochains. Furthermore, generalizing the anti-com\-mutativity of the cup-product, Massey products satisfy certain \f{shuffle identities}, cf.\ \cite{JGMasseyMild}, Prop.\ 4.8. 
\end{rema}

Now assume that the pro-$p$-group $G$ is finitely generated, i.e.\ $h^1(G) < \infty$. It has been remarked by H.\ Koch that higher Massey products are closely related to structure of relation systems of $G$. In order to make this connection precise, we need the definition of the \f{Zassenhaus filtration}. Let $\Omega_G$ denote the complete group algebra 
\[
\Omega_G=\grg=\prlim{U} \F_p[G/U]
\]
where $U$ runs through the open normal subgroups of $G$. By $I_G\subseteq\Omega_G$ we denote the \f{augmentation ideal} of $G$, i.e.\ the kernel of the canonical \f{augmentation map}
\[
 \begin{tikzpicture}[description/.style={fill=white,inner sep=2pt}, bij/.style={above,sloped,inner sep=1.5pt}, column 1/.style={anchor=base east},column 2/.style={anchor=base west}]
 \matrix (m) [matrix of math nodes, row sep=0em,
 column sep=2.5em, text height=1.5ex, text depth=0.25ex]
 {\Omega_G & \F_p,\\
  g & 1,& g\in G.\\
};
 \path[->>,font=\scriptsize]
 (m-1-1) edge node[auto] {} (m-1-2);
 \path[|->,font=\scriptsize]
 (m-2-1) edge node[auto] {} (m-2-2);
 \end{tikzpicture}
\]

The \f{Zassenhaus filtration} $(G_{(n)})_{n\in\N}$ of $G$ is given by
\[
 G_{(n)}=\{g\in G\ |\ g-1\in I_G^n\}.
\]
The groups $G_{(n)}, n\in\N$ form a system of neighborhoods $1\in G$ consisting of open normal subgroups. 

Let
\[
\begin{tikzpicture}[description/.style={fill=white,inner sep=2pt}, bij/.style={below,sloped,inner sep=1.5pt}]
\matrix (m) [matrix of math nodes, row sep=1.5em,
column sep=2.5em, text height=1.5ex, text depth=0.25ex]
{  1 & R & F & G & 1\\};
\path[->,font=\scriptsize]
(m-1-1) edge node[auto] {} (m-1-2)
(m-1-2) edge node[auto] {} (m-1-3)
(m-1-3) edge node[auto] {} (m-1-4)
(m-1-4) edge node[auto] {} (m-1-5);
\end{tikzpicture}
\]
be a minimal presentation of $G$, i.e.\ $F$ is a free pro-$p$-group on generators $x_1, \ldots, x_d,\ d=h^1(G)$. By $\chi_i=x_i^*$ we denote the basis of the $\F_p$-vector space $H^1(F)=H^1(G)$ dual to the $x_i$. By the Hochschild-Serre spectral sequence, we have the transgression isomorphism
\[
 \begin{tikzpicture}[description/.style={fill=white,inner sep=2pt}, bij/.style={above,sloped,inner sep=1.5pt}, column 1/.style={anchor=base east},column 2/.style={anchor=base west}]
 \matrix (m) [matrix of math nodes, row sep=1.5em,
 column sep=2.5em, text height=1.5ex, text depth=0.25ex]
 {\tg:\ H^1(R)^G  & H^2(G).\\
};
 \path[->,font=\scriptsize]
 (m-1-1) edge node[bij] {$\sim$} (m-1-2);
 \end{tikzpicture}
\]
Hence, every element $r \in R$ gives rise to a \f{trace map}
\[
 \begin{tikzpicture}[description/.style={fill=white,inner sep=2pt}, bij/.style={above,sloped,inner sep=1.5pt}, column 1/.style={anchor=base east},column 2/.style={anchor=base west}]
 \matrix (m) [matrix of math nodes, row sep=0em,
 column sep=2.5em, text height=1.5ex, text depth=0.25ex]
 {\tr_r:\ H^2(G) & \mathbb{F}_p,\\
  \varphi& (\tg^{-1}\varphi)(r).\\
};
 \path[->,font=\scriptsize]
 (m-1-1) edge node[auto] {} (m-1-2);
 \path[|->,font=\scriptsize]
 (m-2-1) edge node[auto] {} (m-2-2);
 \end{tikzpicture}
\]
If $r_i\in R, i\in I$ is a minimal system of defining relations for $G$ (i.e.\ a minimal system of generators of $R$ as closed normal subgroup of $F$), then $\{\tr_{r_i}, i\in I\}$ is a basis of the dual space $H^2(G)^\vee$. Furthermore, note that for the free pro-$p$-group $F$ we have a topological isomorphism
\[
 \begin{tikzpicture}[description/.style={fill=white,inner sep=2pt}, bij/.style={above,sloped,inner sep=1.5pt}, column 1/.style={anchor=base east},column 2/.style={anchor=base west}]
 \matrix (m) [matrix of math nodes, row sep=0em,
 column sep=2.5em, text height=1.5ex, text depth=0.25ex]
 { \Omega_F  & \power,\ x_i &1+X_i\\
};
 \path[->,font=\scriptsize]
 (m-1-1) edge node[bij] {$\sim$} (m-1-2);
 \path[|->,font=\scriptsize]
 (m-1-2) edge node[auto] {} (m-1-3);
 \end{tikzpicture}
\]
where $\power$ denotes the $\F_p$-algebra of formal power series in the non-com\-muting indeterminates $X=\{X_1,\ldots, X_d\}$. Let $\psi: F\hookrightarrow \power$ denote the composite of the map $F\hookrightarrow\Omega_F, f\mapsto f-1$ with the above isomorphism, mapping $F$ into the augmentation ideal of $\power$ and the generator $x_i$ to $X_i$.

\begin{defi}
 \label{Magnusdefi}
The element $\psi(f)$ is called \f{Magnus expansion} of $f\in F$. For any multi-index $I=(i_1,\ldots, i_k),\ 1\le i_j\le d$ of height $d$ we set $X_I = X_{i_1}\cdots X_{i_k}$ and define the numbers $\varepsilon_{I,p}(f)$ by
\[
 \psi(f)=\sum_I \varepsilon_{I,p}(f) X_I
\]
where $I$ runs over all multi-indices of height $d$. 
\end{defi}

By definition $f\in F_{(n)}$ holds if and only if $\varepsilon_{I,p}(f)=0$ for all multi-indices $I=(i_1,\ldots, i_k)$ of length $k<n$. The following result proven independently by M.\ Morishita (\cite{MM2}, Th.2.2.2) and D.\ Vogel (\cite{DV2}, Prop.1.2.6) generalizes a well-known connection between the cup-product and relations modulo $F_{(3)}$ to higher degrees:

\begin{theo}
\label{epsilonmaps}
Let $G$ be a finitely presented pro-$p$-group and
\[
\begin{tikzpicture}[description/.style={fill=white,inner sep=2pt}, bij/.style={below,sloped,inner sep=1.5pt}]
\matrix (m) [matrix of math nodes, row sep=1.5em,
column sep=2.5em, text height=1.5ex, text depth=0.25ex]
{  1 & R & F & G & 1\\};
\path[->,font=\scriptsize]
(m-1-1) edge node[auto] {} (m-1-2)
(m-1-2) edge node[auto] {} (m-1-3)
(m-1-3) edge node[auto] {} (m-1-4)
(m-1-4) edge node[auto] {} (m-1-5);
\end{tikzpicture}
\]
be a minimal presentation. Assume that $R\subseteq F_{(n)}$ for some $n\ge 2$. Then for all $2\le k\le n$ the $k$-fold Massey product
\[
 \langle \cdot, \ldots, \cdot\rangle_k: H^1(G)^k\longrightarrow H^2(G)
\]
is uniquely defined. Furthermore, for all multi-indices $I$ of height $d$ and length $2\le |I|\le n$ and for all $r\in R$ we have the equality
\[
 \varepsilon_{I,p}(r) = (-1)^{|I|-1} \tr_r \langle \chi_I\rangle_{|I|}
\]
where for $I=(i_1, \ldots, i_k)$ we have set $\chi_I = (\chi_{i_1},\ldots, \chi_{i_k})\in H^1(G)^k$.  In particular, for $1 < k < n$ the $k$-fold Massey product on $H^1(G)$ is identically zero.
\end{theo}

The above result gives rise to the following

\begin{defi}
\label{zassinvdefi}
 Let $G$ be a finitely generated pro-$p$-group. We define the \f{Zassenhaus invariant} $\z(G)\in\N\cup\{\infty\}$ to be the supremum of all natural numbers $n$ satisfying one of the following equivalent conditions:
\begin{itemize}
 \item[\rm (i)] If $1\to R\to F\to G\to 1$ is a minimal presentation of $G$, then $R\subseteq F_{(n)}$.
 \item[\rm (ii)] The $k$-fold Massey product $H^1(G)^k\to H^2(G)$ is uniquely defined and identically zero for $2\le k< n$.
\end{itemize}
\end{defi}

Note that $\z(G)=\infty$ if and only if $G$ is free. Now assume that $G$ is \f{finitely presented}, i.e.\ $h^1(G),h^2(G)<\infty$. Let 
\[
G=\langle x_1,\ldots, x_d\ |\ r_1,\ldots, r_m\rangle
\]
be a minimal presentation of $G$, i.e.\ $G=F/R$ where $F$ is the free pro-$p$-group on generators $x_1,\ldots,x_d,\ d=h^1(G)$ and $R$ is generated by $r_1,\ldots, r_m$ as closed normal subgroup of $F$. Let $\rho_i,\ i=1,\ldots, m$ denote the initial form of $r_i$ in the graded $\F_p$-Lie algebra
\[
 \gr F = \bigoplus_{n\ge 1} F_{(n)} / F_{(n+1)}.
\]
The map $\psi: F\hookrightarrow\power$ induces a inclusion (of $\F_p$-Lie algebras)
\[
 \gr F\subseteq \gr \power = \bigoplus_{n\ge 0} \power_n / \power_{n+1} = \polyx
\]
where $\power_n\subseteq \power$ denotes the two-sided ideal of power series of degree $\ge n$ and $\polyx$ is the free associative $\F_p$-algebra on $X=\{X_1,\ldots, X_d\}$. The presentation $G=\langle x_1,\ldots, x_d\ |\ r_1,\ldots, r_m\rangle$ is called \f{strongly free}, if the images of $\rho_i$ in $\polyx$ form a \f{strongly free sequence}. (The notion of strongly free sequences is due to D.\ Anick, see \cite{JGMasseyMild}, Def.2.7 for a definition.) We say that $G$ is a \f{mild pro-$p$-group} (with respect to the Zassenhaus filtration) if it possesses a strongly free presentation.

\begin{rema}
 The theory of mild groups has originally been developed by J.\ Labute in the case of discrete groups and later applied to pro-$p$-groups in \cite{JL}.  Note that they can be defined with respect to different filtrations such as weighted lower $p$-central series and weighted Zassenhaus filtrations.
\end{rema}

The main properties of mild pro-$p$-groups is given by the following

\begin{theo}
 Let $G$ be a mild pro-$p$-group and $G=\langle x_1,\ldots, x_d\ |\ r_1, \ldots, r_m\rangle$ a strongly free presentation. Then $h^2(G)=m$, $G$ is of cohomological dimension $\cd G=2$ and if $m\neq d-1$ it is not $p$-adic analytic.
\end{theo}

For a proof we refer to \cite{JGMasseyMild}, Th.2.12 which is a slight generalization of \cite{JL}, Th.5.1.

In order to find arithmetical examples of mild pro-$p$-groups $G$ with Zassenhaus invariant $\z(G)=3$, we will make use of the following result, cf.\ \cite{JGMasseyMild} Th.4.9:

\begin{theo}
\label{cohomologicalcrit}
 Let $p$ be a prime number and $G$ a finitely presented pro-$p$-group with $n=\z(G)<\infty$. Assume that $H^1(G)$ admits a decomposition $H^1(G)=U\oplus V$ as $\F_p$-vector space such that for some natural number $e$ with $1\le e \le n-1$ the $n$-fold Massey product $\langle\cdot, \ldots, \cdot \rangle_n: H^1(G)^n \longrightarrow H^2(G)$ satisfies the following conditions:
\begin{itemize}
 \item[\rm (a)] $\langle \xi_1,\ldots, \xi_n \rangle_n =0$ for all tuples $(\xi_1,\ldots, \xi_n)\in H^1(G)^n$ such that $\#\{i\ |\ \xi_i\in V\}\ge n-e+1$.
 \item[\rm (b)] $\langle\cdot, \ldots, \cdot \rangle_n$ maps 
\[
U^{\otimes e} \otimes V^{\otimes n-e}
\]
surjectively onto $H^2(G)$.
\end{itemize}
Then $G$ is mild (with respect to the Zassenhaus filtration).
\end{theo}

\section{Totally real pro-$2$-extensions with wild ramification}

As has been remarked in the introduction, the theory of mild pro-$p$-groups has had a great impact in the study of pro-$p$-extensions of number fields with restricted ramification. From the group-theoretical point of view, an important tool in the proof of Schmidt's theorem \cite{AS2}, Th.1.1 is the \f{cup-product criterion}. The groups $G_{S\cup S_0}(p)$ that are obtained by enlarging the set $S$ and satisfy $\cd G_{S\cup S_0}(p)=2$ have Zassenhaus invariant $\z(G_{S\cup S_0}(p))=2$. 

On the other hand, for sets of primes $S$ such that $\z(G_S(p))\ge 3$ there have been no known examples where the cohomological dimension of $G_S(p)$ is finite. Having the generalized criterion \ref{cohomologicalcrit} using higher Massey products at hand, the question naturally arises whether mild pro-$p$-groups with Zassenhaus invariant at least $3$ arise as arithmetically defined Galois groups, i.e.\ (quotients of) groups of the form $G_S(p)$. In the following we give a positive answer in the case of pro-$2$-extensions of the rationals.
\vspace{10pt}

We fix the following notation: For a finite set $S$ of primes of $\Q$, we denote by $\Q_S(2)|\Q$ the maximal pro-$2$-extension of $\Q$ unramified outside $S$ and denote by $G_S(2)=\Gal(\Q_S(2)|\Q)$ its Galois group. Let $\infty$ denote the infinite prime of $\Q$. We consider the local extension $\C|\R$ as being ramified, i.e.\ if $\infty\not\in S$, $\Q_S(2)|\Q$ is totally real.

An explicit construction of mild groups of the form $G_S(2)$ with Zassenhaus invariant $3$ amounts to giving an arithmetic description of the triple Massey product
\[
\langle \cdot,\cdot,\cdot \rangle_3: H^1(G_S(2)) \times H^1(G_S(2)) \times H^1(G_S(2)) \longrightarrow H^2(G_S(2)),
\]
or, equivalently, of the relation structure of $G_S(2)$ modulo the fourth step of the Zassenhaus filtration. Such a description has been given by M.\ Morishita and D.\ Vogel in terms of \f{Rédei symbols} in the case where $S$ is a finite set consisting of the infinite prime and prime numbers $l_1,\ldots, l_n\equiv 1\mod 4$ such that the Legendre symbols 
$\left(\frac{l_i}{l_j}\right)$ are pairwise trivial. However, for such sets $S$ the extension $\Q_S(2)$ is always totally imaginary and therefore the Galois group $G_S(2)$, having $2$-torsion, satisfies $\cd G_S(2)=\infty$. Consequently, in order to construct mild examples, we have to remove ramification at $\infty$. 

By results of I.R.\ \v{S}afarevi\v{c} and H.\ Koch, there is a presentation for $G_S(2)$ in terms of generators and relations of local Galois groups provided that the \f{\v{S}afarevi\v{c}-Tate group} $\Sha_S(2)$ given by the exact sequence
\[
\begin{tikzpicture}[description/.style={fill=white,inner sep=2pt}, bij/.style={above,sloped,inner sep=1.5pt}]
\matrix (m) [matrix of math nodes, row sep=3em,
column sep=2.5em, text height=1.5ex, text depth=0.25ex]
{ 0 & \Sha_S(2) & H^2(G_S(2),\F_2) & \prod_{l\in S} H^2(G_l(2),\F_2)\\};
\path[->,font=\scriptsize]
(m-1-1) edge node[auto] {} (m-1-2)
(m-1-2) edge node[auto] {} (m-1-3)
(m-1-3) edge node[auto] {} (m-1-4);
\end{tikzpicture}
\]                                                                                                                                                                            
vanishes where $G_l(2)$ denotes the Galois group of the maximal pro-$2$-extension of $\Q_l$ if $l$ is a finite prime and $G_l(2)=\Gal(\C|\R)\cong \Z/2\Z$ if $l=\infty$. We have a natural inclusion $\Sha_S(2)\hookrightarrow \rusb_S(2)$ into the Pontryagin dual $\rusb_S(2)=(V_S(2))^\vee$ of the \f{Kummer group}
\[
 V_S(2)=\{x\in \Q^\times|\ x\in U_l \Q_l^{\times 2}\ \mbox{if}\ l\not\in S,\ x\in\Q_l^{\times 2}\ \mbox{if}\ l\in S\}/\Q^{\times 2}
\]
where $U_l=\Z_l^\times$ if $l\neq\infty$ and $U_l=\R^\times$ if $l=\infty$. Noting that 
\[
V_\varnothing(2)=\{\pm 1\} \Q^{\times 2}/\Q^{\times 2}\cong \{\pm 1\},
\]
we see that if $\infty\not\in S$ then $V_S(2)=1$ if and only if $S$ contains $2$ or a prime $l\equiv 3\mod 4$. However if $S$ contains more than one prime $l\equiv 3\mod 4$, we have $\z(G_S(2))=2$ since by quadratic reciprocity there are non-trivial Legendre symbols. If $S$ contains exactly one prime $l\equiv 3\mod 4$ then it can be shown that $G_S(2)$ is never mild. Consequently in the following we consider a finite $S$ of rational primes such that $2\in S, \infty\not\in S$. Recall that for a pro-$2$-group $G$ we set $H^i(G)=H^i(G,\F_2),\ h^i(G)=\dim_{\F_2} H^i(G)$. We start with the following

\begin{prop}
\label{gsrank}
Let $S=\{2, l_1, \dots, l_n\}$ for some $n\ge 1$ and odd prime numbers $l_1, \ldots, l_n$. Then the pro-$2$-group $G_S(2)$ has cohomological dimension $\cd G_S(2)=2$ and the generator and relation ranks satisfy 
\[
h^1(G_S(2)) = n+1, \ h^2(G_S(2)) = n.
\]
Furthermore, the abelianization $G_S(2)^{ab}$ is infinite.
\end{prop}
\begin{proof}
Since by assumption $2\in S, \infty\not\in S$, \cite{NSW} Th.10.6.1 yields $\cd G_S(2)\le 2$ and $\chi_2 (G_S(2)) = 0$ where $\chi_2 (G_S(2))=\sum_{i\ge 0} (-1)^i h^i(G_S(2))$ denotes the Euler-Poincaré characteristic of $G_S(2)$. Since $\rusb_S(\Q)=(V_S(\Q))^\vee$ is trivial, the general formula \cite{NSW}, Th.10.7.12 yields $h^1(G_S(2)) = 1+n$ and thus also the second formula $h^2(G_S(2))=n$ holds. In particular, $\cd G_S(2)=2$ and we have $h^2(G_S(2))< h^1(G_S(2))$, which implies the infiniteness of $G_S(2)^{ab}$. Alternatively, the infiniteness of $G_S(2)^{ab}$ also follows directly from the fact that $\Q_S(2)$ contains the cyclotomic $\Z_2$-extension of $\Q$. 
\end{proof}

We will now determine necessary and sufficient conditions for $\z(G_S(2))>2$. By the results of \v{S}afarevi\v{c} and Koch, we have a minimal presentation of groups of the form $G_S(p)$ with the relations determined modulo the third step of the Zassenhaus filtration. Choosing generators of inertia subgroups of $G_S(2)$ and appropriate relations of corresponding local Galois groups, this is made possible by the interplay of local and global class field theory. For the latter we make use of the idèlic formulation throughout, cf.\ \cite{NSW}, Ch.VIII. 
\vspace{10pt}

The following lemma is well-known. However, as has been pointed out to the author by D.\ Vogel, the main reference \cite{HK} contains a sign error in the case $l=2$, therefore we included the statement here:

\begin{lemm}
\label{localgaloisgroups}
Let $l$ be a prime number and $G_l(2)$ be the Galois group of the maximal $2$-extension $\Q_l(2)$ of $\Q_l$. Let $\mathcal{T}_l(2)\subseteq G_l(2)$ denote the inertia group of $G_l(2)$.
\begin{itemize}
\item[\rm (i)] If $l\neq 2$, then $\mathcal{T}_l(2)\cong \Z_2$ and $G_l(2)$ is a pro-$2$-group with two generators $\sigma, \tau$ satisfying the relation
\[
\tau^{l-1} [\tau^{-1}, \sigma^{-1}]=1
\]
where $\sigma$ denotes an arbitrary lift of the Frobenius automorphism of the maximal unramified $2$-extension of $\Q_l$ and $\tau$ is an arbitrary generator of $\mathcal{T}_l$.
 \item[\rm (ii)] If $l=2$, let $\sigma, \tau, \tilde{\tau}$ be arbitrary elements of $G_2(2)$ such that
\begin{eqnarray*}
\sigma & \equiv & (2,\Q_2(2)^{ab}|\Q_2)\mod [G_2(2), G_2(2)],\\
\tau & \equiv & (5,\Q_2(2)^{ab}|\Q_2)\mod [G_2(2), G_2(2)],\\
\tilde{\tau} & \equiv & (-1,\Q_2(2)^{ab}|\Q_2)\mod [G_2(2), G_2(2)]\\
\end{eqnarray*}
where $(\ \cdot\ , \Q_2(2)^{ab}|\Q_2)$ denotes the local norm residue symbol. Then $\sigma$ is a lift of the Frobenius automorphism of the maximal unramified $2$-extension of $\Q_2$, $\tau, \tilde{\tau}\in \mathcal{T}_2(2)$ and $\{\sigma, \tau, \tilde{\tau}\}$ form a minimal system of generators of $G_2(2)$. We have $h^2(G_2(2)) = 1$ and in a minimal presentation 
\[
\begin{tikzpicture}[description/.style={fill=white,inner sep=2pt}, bij/.style={below,sloped,inner sep=1.5pt}]
\matrix (m) [matrix of math nodes, row sep=3em,
column sep=2.5em, text height=1.5ex, text depth=0.25ex]
{1 & R & F & G_2(2) & 1\\};
\path[->,font=\scriptsize]
(m-1-1) edge node[auto] {} (m-1-2)
(m-1-2) edge node[auto] {} (m-1-3)
(m-1-3) edge node[auto] {} (m-1-4)
(m-1-4) edge node[auto] {} (m-1-5);
\end{tikzpicture}
\]
with preimages $s,t, \tilde{t}$ of $\sigma, \tau, \tilde{\tau}$ respectively, the single relation $r$ can be chosen in the form
\[
r\equiv \tilde{t}^{\phantom{.}2}[t,s]\mod F_{(3)}.
\]
Furthermore, $\tau$ and $\tilde{\tau}$ generate $\mathcal{T}_2(2)$ as a normal subgroup of $G_2(2)$.
\end{itemize}
\end{lemm}
\begin{proof}
The first statement is a special case of \cite{HK}, Th.10.2. For the second statement note that for $\pi=2,\ \alpha_0=5,\ \alpha_1=-1$ the set $\{\pi, \alpha_0, \alpha_1\}$ is a basis of $\Q_2^\times/\Q_2^{\times 2}$ satisfying the conditions of \cite{HK}, Lemma 10.10. (Note that with a view to the correctness of the following Th.10.12 loc. cit., this lemma contains a sign error in (iv); the condition $(\alpha_1,\pi) = -1$ has to be replaced by $(\alpha_1,\pi) = 1$, which holds for the above choice.) Then \cite{HK}, Th.10.12 yields (ii). In fact, in order to see that $\tau, \tilde{\tau}$ generate $\mathcal{T}_2(2)$ as a normal subgroup of $G_2(2)$, let $\Gamma_2(2)=G_2(2)/\mathcal{T}_2(2)$ and consider the Hochschild-Serre exact sequence
\[
\begin{tikzpicture}[description/.style={fill=white,inner sep=2pt}, bij/.style={below,sloped,inner sep=1.5pt}]
\matrix (m) [matrix of math nodes, row sep=3em,
column sep=2.3em, text height=1.5ex, text depth=0.25ex]
{0 & H^1(\Gamma_2(2)) & H^1(G_2(2)) & H^1(\mathcal{T}_2(2))^{\Gamma_2(2)} & H^2(\Gamma_2(2)),\\};
\path[->,font=\scriptsize]
(m-1-1) edge node[auto] {} (m-1-2)
(m-1-2) edge node[auto] {} (m-1-3)
(m-1-3) edge node[auto] {} (m-1-4)
(m-1-4) edge node[auto] {} (m-1-5);
\end{tikzpicture}
\]
which yields $\dim_{\F_2} H^1(\mathcal{T}_2(2))^{\Gamma_2(2)}=2$ since $h^1(\Gamma_2(2))=1,\ h^2(\Gamma_2(2))=0$. The elements $\tau, \tilde{\tau}\in\mathcal{T}_2(2)$ are contained in a minimal system of generators of $G_2(2)$, hence they are linearly independent modulo $\mathcal{T}_2(2)^2[G_2(2),\mathcal{T}_2(2)]$. Now the claim follows by reason of dimension.
\end{proof}

We return to the global Galois group $G_S(2)$. Let $S=\{l_0,l_1,\ldots, l_n\}$ where $l_0=2$ and $l_i$ is an odd prime number for $i=1,\ldots, n$. We fix the following notations:

\begin{itemize}
 \item[\rm (i)] For each $0\le i\le n$ let $\mf{l}_i$ denote a fixed prime of $\Q_S(2)$ above $l_i$.
 \item[\rm (ii)] For $1\le i\le n$ let $\hat{l}_i$ denote the idèle of $\Q$ whose $l_i$-component equals $l_i$ and all other components are $1$ and let $\hat{g}_i$ denote the idèle whose $l_i$-component equals $g_i$ for a primitive root $g_i$ modulo $l_i$ and all other components are $1$.
 \item[\rm (iii)] For $i=0$ let $\hat{l}_0,\ \hat{g}_0,\ \hat{g}'_0$ denote the idèles of $\Q$ whose $2$-components are $2,5$ and $-1$ respectively and all other components are $1$.
\end{itemize}

For $0\le i\le n$ we choose an element $\sigma_i\in G_S(2)$ with the following properties:
\begin{itemize}
 \item[\rm (i)] $\sigma_i$ is a lift of the Frobenius automorphism of $\mf{l}_i$ with respect to the maximal subextension of $\Q_S(2)|\Q$ in which $\mf{l}_i$ is unramified;
 \item[\rm (ii)] the restriction of $\sigma_i$ to the maximal abelian subextension $\Q_S(2)^{ab}|\Q$ of $\Q_S(2)|\Q$ equals $(\hat{l}_i,\Q_S(2)^{ab}|\Q)$ where $(\ \cdot\ , \Q_S(2)^{ab}|\Q)$ denotes the global norm residue symbol.
\end{itemize}

For $1\le i\le n$ we denote by $T_{\mf{l}_i}\subseteq G_S(2)$ the inertia subgroup of $\mf{l}_i$ and choose an element $\tau_i\in T_{\mf{l}_i}$, such that
\begin{itemize}
 \item[\rm (i)] $\tau_i$ generates $T_{\mf{l}_i}$;
 \item[\rm (ii)] the restriction of $\tau_i$ to $\Q_S(2)^{ab}|\Q$ equals $(\hat{g}_i,\Q_S(2)^{ab}|\Q)$.
\end{itemize}

Finally, let $\tau_0,\tilde{\tau}_0$ denote two elements of the inertia subgroup $T_{\mf{l}_0}\subseteq G_S(2)$ of $\mf{l}_0$ such that 
\begin{itemize}
 \item[\rm (i)] $\tau_0,\tilde{\tau}_0$ generate $T_{\mf{l}_0}$ as a normal subgroup of the decomposition group $G_{\mf{l}_0}\subseteq G_S(2)$ of $\mf{l}_0$;
 \item[\rm (ii)] the restriction of $\tau_0$ to $\Q_S(2)^{ab}|\Q$ equals $(\hat{g}_{0},\Q_S(2)^{ab}|\Q)$ and the restriction of $\tilde{\tau}_0$ to $\Q_S(2)^{ab}|\Q$ equals $(\hat{g}'_{0},\Q_S(2)^{ab}|\Q)$.
\end{itemize}

The existence of the elements $\sigma_i, \tau_i, \tilde{\tau}_0$ follows by class field theory and the structure result \ref{localgaloisgroups} of local pro-$2$-Galois groups. 

\begin{defi}
For $1\le i,j\le n, \ i\neq j$ we define the \f{linking number} $a_{i,j}\in\F_2$ by
\[
 a_{i,j}=\left\{ \begin{array}{ll} 
1,& \mbox{if}\ \left(\frac{l_i}{l_j}\right)_2=-1 ,\\
0,& \mbox{else}.\\
\end{array}\right.
\]
Furthermore, for $1\le i\le n$ we define the numbers $a_{i,0},\tilde{a}_{i,0}\in\F_2$ by
 \begin{eqnarray*}
 a_{i,0}&=&\left\{ \begin{array}{ll} 
1,& \mbox{if}\ l_i\equiv 3,5\mod 8,\\
0,& \mbox{else}\\
\end{array}\right.\quad and\\ 
 \tilde{a}_{i,0}&=&\left\{ \begin{array}{ll} 
1,& \mbox{if}\ l_i\equiv 3,7\mod 8,\\
0,& \mbox{else}.\\
\end{array}\right.
\end{eqnarray*}
\end{defi}

We can now give the desired description of the group $G_S(2)$ in terms of generators and relations:

\begin{prop}
\label{gspres}
 Let $S=\{l_0,\ldots, l_n\}$ where $l_0=2$ and $l_1,\ldots, l_n$ are odd prime numbers. Furthermore, let $F$ be the free pro-$2$-group on the $n+1$ generators $x_0, \ldots, x_n$. Then $G_S(2)$ admits a minimal presentation
\[
\begin{tikzpicture}[description/.style={fill=white,inner sep=2pt}, bij/.style={below,sloped,inner sep=1.5pt}]
\matrix (m) [matrix of math nodes, row sep=3em,
column sep=2.5em, text height=1.5ex, text depth=0.25ex]
{1 & R & F & G_S(2) & 1\\};
\path[->,font=\scriptsize]
(m-1-1) edge node[auto] {} (m-1-2)
(m-1-2) edge node[auto] {} (m-1-3)
(m-1-3) edge node[auto] {$\pi$} (m-1-4)
(m-1-4) edge node[auto] {} (m-1-5);
\end{tikzpicture}
\]
where $\pi$ maps $x_i$ to $\tau_i$ for $0\le i\le n$. Let $y_i$ denote a preimage of $\sigma_i$ under $\pi$, then a minimal generating set of $R$ as closed normal subgroup of $F$ is given by
\[
 r_i = x_i^{l_i-1}[x_i^{-1}, y_i^{-1}],\ i=1, \ldots, n
\]
and we have
\begin{eqnarray*}
r_i\equiv x_i^{l_i-1} \prod_{\stackrels{0\le j\le n}{j\neq i}} [x_i,x_j]^{a'_{i,j}}\ \mod F_{(3)}
\end{eqnarray*}
where the numbers $a'_{ij}\in\F_2$ are given by
\[
 a'_{i,j}=\left\{\begin{array}{ll} 
a_{i,j} + \tilde{a}_{i,0},& \mbox{if}\ l_j\equiv 3\modulo 4,\\
a_{i,j},& \mbox{else.}\\
\end{array}\right.
\]
\end{prop}
\begin{proof}
This is a special case of \cite{HK}, Th.11.10. In fact, first note that since $2\in S$, we have $\Sha_S(2)=\rusb_S(2)=0$. By class field theory and \ref{gsrank} it follows that the set $\{\tau_0, \tau_1,\ldots, \tau_n\}$ is a minimal system of generators of $G_S(2)$. A system of defining relations is given by the local relations
\begin{eqnarray*}
 r_0&=& \tilde{x}_0^2 [x_0, y_0] r_0',\\
 r_i&=&x_i^{l_i-1}[x_i^{-1}, y_i^{-1}],\ i=1, \ldots, n
\end{eqnarray*}
for some $r_0'\in F_{(2)}$ where $\tilde{x}_0$ denotes a preimage of $\tilde{\tau}_0$ under $\pi$, cf.\ \ref{localgaloisgroups}. Any of these relations may be omitted and we decide to ignore $r_0$. According to the choices we have made and by definition of the linking numbers $a_{i,j}$ the elements $y_i$ satisfy
\[
 y_i\equiv x_0^{a_{i,0}} \tilde{x}_0^{\tilde{a}_{i,0}}\prod_{\stackrels{1\le j\le n}{j\neq i}} x_j^{a_{i,j}}\ \mod F_{(2)}.
\]
Class field theory implies
\[
 \tilde{x}_0 \equiv \prod_{\stackrels{1\le j\le n,}{l_j\equiv 3\modulo 4}} x_j \ \mod F_{(2)}
\]
which finishes the proof.
\end{proof}

\begin{coro}
\label{cupproductcoro}
 Let the set $S$ be given as in \ref{gspres}. For the Zassenhaus invariant of $G_S(2)$ we have $\z(G_S(2))\ge 3$ if and only if $l_i\equiv 1\modulo 8,\ i=1,\ldots, n$ and the Legendre symbols satisfy
\[
 \left(\frac{l_i}{l_j}\right)_2=1,\ 1\le i,j\le n,\ i\neq j.
\]
\end{coro}
\begin{proof}
 Keeping the notation of \ref{gspres}, the cup-product vanishes if and only if $R\subseteq F_{(3)}$, i.e.\ $r_i\in F_{(3)}$ for $1\le i\le n$. The latter is equivalent to $x_i^{l_i-1} \in F_{(3)}$ and $a'_{i,j}=0$ for $1\le i\le n,\ 0\le j\le n,\ i\neq j$. Noting that for $1\le i\le n$ we have $x_i^{l_i-1}\in F_{(3)}$ if and only if $l_i\equiv 1\modulo 4$, the claim follows immediately from the definition of the numbers $a'_{i,j}$.
\end{proof}

\section{Rédei symbols}

We fix a set 
\[
 S=\{2,l_1,\ldots, l_n\}
\]
where $l_i\equiv 1\modulo 8,\ i=1,\ldots, n$ are pairwise distinct prime numbers with Legendre symbols $\left(\frac{l_i}{l_j}\right)=1,\ 1\le i,j\le n,\ i\neq j$.

By \ref{cupproductcoro}, there exists a uniquely defined triple Massey product
\[
 H^1(G_S(2))\times H^1(G_S(2))\times  H^1(G_S(2))\longrightarrow H^2(G_S(2))
\]
determining the relation structure of $G_S(2)$ modulo the fourth step of the Zassenhaus filtration. Under some further conditions on the primes in $S$, we will give an explicit description of this product in terms of certain arithmetical symbols called \f{Rédei symbols}. These symbols have first been introduced in \cite{LR} in order to study diophantine equations.
\vspace{10pt}

For the definition of the Rédei symbol $[\cdot,\cdot,\cdot]$, we need the following notation:

\begin{defi}
 Let $k$ be a number field, $\alpha\in k^\times\setminus k^{\times 2}$ and let $\mf{p}$ be a prime ideal of $k$. Then we set
 \[
 \left(\frac{\alpha|k}{\mathfrak{p}}\right)=\left\{ \begin{array}{ll}
\phantom{-}1,& \mbox{if}\ \mathfrak{p}\ \mbox{splits in}\ k(\sqrt{\alpha}),\\
\phantom{-}0,& \mbox{if}\ \mathfrak{p}\ \mbox{is ramified in}\ k(\sqrt{\alpha}),\\
-1,& \mbox{if}\ \mathfrak{p}\ \mbox{is inert in}\ k(\sqrt{\alpha}).\\
\end{array}\right.
\]
\end{defi}

\begin{prop}
\label{normproposition}
Let $0\le i,j\le n$ such that $(i,j)\neq (0,0)$. Let $k_1$ denote the real quadratic number field $k_1=\Q(\sqrt{l_i})$. Then there exists an element $\alpha\in k_1$ satisfying the following properties:
\begin{itemize}
 \item[(1)] $N_{k_1|\Q}(\alpha) = l_j z^2$ for some $z\in\Z$ where $N_{k_1|\Q}(\cdot)$ denotes the norm,
 \item[\rm (2)] $N_{k_1|\Q}(D_{k_{12}|k_1})= l_j$ if $j\neq 0$ or $N_{k_1|\Q}(D_{k_{12}|k_1})= 8$ if $j=0$  where $k_{12}=k_1(\sqrt{\alpha})$ and $D_{k_{12}|k_1}$ denotes the discriminant of the extension $k_{12}|k_1$.
\end{itemize}
In addition, let $0\le k\le n$ such that $(l_i,l_j,l_k)=1$. Then there exists a prime ideal $\mf{p}$ in $k_1$ above $l_k$ which is unramified in $k_{12}$. For all such choices of $\alpha$ and $\mf{p}$, the symbol $\left(\frac{\alpha|k_1}{\mathfrak{p}}\right)$ yields the same (non-zero) value.
\end{prop}
\begin{proof}
 This is a special case of \cite{LR}, Satz 1. In fact, first suppose $j\neq 0$. Using the assumptions we made for the primes $l_i$, the classical Hasse-Minkowski theorem implies the existence of integers $x,y,z\in\Z$ satisfying the equation
\[
 x^2-l_iy^2-l_jz^2 = 0,
\]
such that $(x,y,z)=1,\ x+y\sqrt{l_i}\equiv 1\mod 4\mc{O}_{k_1}$ where $\mc{O}_{k_1}$ denotes the ring of integers of $k_1=\Q(\sqrt{l_i})$. Then $\alpha= x+y\sqrt{l_i}\in k_1$ satisfies the conditions $(1)$ and $(2)$ (whereas the first condition is obvious, the second one requires a thorough examination of the discriminant for which we refer the reader to \cite{LR}). Now assume that $i\neq 0$ and $\alpha'\in k_2=\Q(\sqrt{l_j})$ satisfies conditions (1) and (2) after replacing $l_j$ by $l_i$. Then again by \cite{LR} the element
\[
\alpha= Tr_{k_2|\Q} (\alpha') + 2\sqrt{N_{k_2|\Q}(\alpha')}=(\sqrt{\alpha'} + \sqrt{\overline{\alpha'}})^2\in k_1
\]
satisfies conditions (1) and (2). For the second part of the proposition we again refer to \cite{LR}, Satz 1.
\end{proof}

\begin{defi}
\label{Redeidefi}
Let $0\le i,j, k\le n$ such that $(i,j)\neq (0,0)$. Keeping the notations and assumptions of \ref{normproposition}, the \f{Rédei symbol} $[l_i,l_j,l_k]\in\{\pm1\}$ is defined by
\[
 [l_i, l_j, l_k] = \left(\frac{\alpha|k_1}{\mf{p}}\right).
\]
Furthermore, we set
 \[
 [l_0,l_0,l_k] = \left\{\begin{array}{ll} 
\ 1,& \mbox{if}\ l_k\equiv 1\mod 16,\\  
\ -1 ,& \mbox{if}\ l_k\equiv 9\mod 16.
\end{array}\right.
\]
\end{defi}

If $(i,j)\neq (0,0)$, by definition the Rédei symbol $[l_i,l_j,l_k]$ describes the decomposition behavior of the prime $l_k$ in the field $\Q(\sqrt{l_i})(\sqrt{\alpha})$. If $i=j=0$, it describes the decomposition behavior of $l_k$ in the maximal real subfield of the cyclotomic field $\Q(\zeta_{16})$. In the original paper \cite{LR}, Rédei proves the remarkable fact that the symbol $[\cdot,\cdot,\cdot]$ is symmetric: For any permutation $\gamma$ of the indices $\{i,j,k\}$ it holds that
\[
 [l_i,l_j,l_k]=[l_{\gamma(i)},l_{\gamma(j)}, l_{\gamma(k)}].
\]
Let $K_{i,j}$ denote the Galois closure of $\Q(\sqrt{l_i})(\sqrt{\alpha})|\Q$. It is also the Galois closure of the field $\Q(\sqrt{l_j})(\sqrt{\alpha'})$ where $\alpha'$ is given as in $\ref{normproposition}$ after exchanging $l_i$ and $l_j$. Furthermore, $\Q(\sqrt{l_i},\sqrt{l_j})\subseteq K_{i,j}$. For the structure of the Galois group $G(K_{i,j}|\Q)$ we have the following result proven in \cite{LR}:

\begin{prop}
\label{discriminantprop}
Keeping the notations and assumptions of \ref{normproposition}, the following holds for $(i,j)\neq (0,0)$:
\begin{itemize}
 \item[\rm (i)] If $l_i\neq l_j$, then $G(K_{i,j}|\Q)$ is the dihedral group of order $8$. Let $s,t$ be generators of $G(K_{i,j}|\Q(\sqrt{l_i})(\sqrt{\alpha}))$ and $G(K_{i,j}|\Q(\sqrt{l_i})(\sqrt{\alpha'}))$ respectively, i.e.
\[
 s: \sqrt{\alpha}\mapsto -\sqrt{\alpha},\quad t: \sqrt{\alpha'}\mapsto -\sqrt{\alpha'}
\]
where $\alpha'$ is given as in $\ref{normproposition}$ after exchanging $l_i$ and $l_j$. Then $G(K_{i,j}|\Q)$ admits the presentation
\[
 G(K_{i,j}|\Q)=\langle s,t\ |\ s^2=t^2=(st)^4 = 1\rangle.
\]
 If $l_i=l_j$, then $K_{i,j}$ is a cyclic extension of degree $4$ over $\Q$.
 \item[\rm (ii)] The discriminant of $K_{a_1,a_2}$ is given by 
\[
 D_{K_{i,j}|\Q}= \left\{ \begin{array}{ll} 
\ol{l_i}^4 \ol{l_j}^4,& \mbox{if}\ l_i\neq l_j,\\
l_i^3,& \mbox{if}\ l_i=l_j
\end{array}\right.
\]
where we have set $\ol{l_i}=l_i$ for $i\neq 0$ and $\ol{l_0} = 8$. In particular, $K_{i,j}|\Q$ is unramified outside the set $\{l_i,l_j,\infty\}$.
\end{itemize}
\end{prop}

\begin{defi}
Keeping the notation of \ref{normproposition}, we say that the pair $(l_i,l_j)$ is \f{totally real} if $i=j=0$ or the element $\alpha$ can be chosen in such a way that $\Q(\sqrt{l_i})(\sqrt{\alpha})$ and hence $K_{i,j}$ is totally real.
\end{defi}

By \ref{discriminantprop}, we have $K_{i,j}\subseteq \Q_S(2)$ if and only if $(l_i,l_j)$ is totally real. We are thus led to the question for sufficient and necessary conditions for this property. First observe that $\Q(\sqrt{l_i})(\sqrt{\alpha})$ is totally real if and only if the Diophantine equation $x^2-l_iy^2-l_jz^2 = 0$ in the proof of \ref{normproposition} admits a solution $(x,y,z)$ with $x>0$. We have the following

\begin{prop}\quad
 \label{totallyrealproposition}
\begin{itemize}
 \item[\rm (i)]
For $1\le i\le n$, the pair $(l_i,l_i)$ is totally real. 
 \item[\rm (ii)] 
For $1\le i\le j\le n,\ i\neq j$ the field $K_{i,j}$ is the unique Galois extension of degree $4$ of the quadratic field $\Q(\sqrt{l_il_j})$ unramified outside the infinite primes. In particular, it is independent of the choice of the element $\alpha$ in \ref{normproposition}. 
 \item[\rm (iii)] For $1\le i\le j\le n,\ i\neq j$ the following assertions are equivalent:
\begin{itemize}
 \item[\rm (1)] The pair $(l_i,l_j)$ is totally real.
 \item[\rm (2)] The class number of $\Q(\sqrt{l_il_j})$ is divisible by 4.
 \item[\rm (3)] The class number of $\Q(\sqrt{l_i},\sqrt{l_j})$ is even.
 \item[\rm (4)] For the fourth power residue symbol $\left(\frac{\cdot}{\cdot}\right)_4$ it holds that
\[
 \left(\frac{l_i}{l_j}\right)_4 = \left(\frac{l_j}{l_i}\right)_4.
\]
\end{itemize}
\end{itemize}
\end{prop}
\begin{proof}
In order to show (i), note that $l_i$ can be written as $l_i=y^2+z^2$ with $y\equiv 0\mod 4$. This gives rise to a solution of the equation $x^2-l_iy^2-l_iz^2$ satisfying the properties in the proof of \ref{normproposition} with $x=l_i>0$ which shows (i).
\vspace{10pt}

For $i\neq j$, a straightforward examination of the subfields of $K_{i,j}$ shows that $K_{i,j}|\Q(\sqrt{l_i l_j})$ is a cyclic extension of degree $4$ unramified at all finite places, cf.\ \cite{LR}. By Gauss' genus theory, the $2$-rank of the narrow class group of $\Q(\sqrt{l_i l_j})$ is $1$ (e.g.\ see \cite{MKClassNumber}) showing (ii).
\vspace{10pt}

If $(l_i,l_j)$ is totally real, $K_{i,j}|\Q(\sqrt{l_i l_j})$ is an unramified extension of degree 4 implying (2). Conversely, assuming (2), the field $\Q(\sqrt{l_i l_j})$ possesses an unramified extension of degree $4$ which by (ii) must coincide with $K_{i,j}$. Hence $(1)$ holds. For the equivalences $(2)\Leftrightarrow(3)\Leftrightarrow(4)$ see \cite{RK}, Th.1.
\end{proof}

We can now state and prove the fundamental relation between the triple Massey products for the group $G_S(2)$ and Rédei symbols. We keep the system of generators $\tau_i$ and relations $r_i$ as chosen in \ref{gspres}.

\begin{theo}
 \label{gsmassey}
Let $S=\{l_0,\ldots, l_n\}$ where $l_0=2$ and $l_1,\ldots, l_n$ are prime numbers $\equiv 1\modulo 8$ satisfying $\left(\frac{l_i}{l_j}\right)_2=1,\ 1\le i,j\le n,\ i\neq j$. Then the group $G_S(2)$ has Zassenhaus invariant $\z(G_S(2))\ge 3$ and the triple Massey product 
\[
\langle\cdot, \cdot, \cdot\rangle_n: H^1(G_S(2))\times H^1(G_S(2))\times H^1(G_S(2))  \longrightarrow H^2(G_S(2))
\]
is given by
\[
(-1)^{tr_{r_m}\langle \chi_{i}, \chi_{j}, \chi_{k} \rangle_3} =
 \left\{ \begin{array}{ll}
[l_i,l_j,l_k],& \mbox{if}\ m=k, m\neq i,\\ 
\redeiinvnotilde,& \mbox{if}\ m=i, m\neq k,\\ 
1,& \mbox{otherwise}
\end{array}\right.
\]
for all $1\le m\le n$ and provided that $(l_i, l_j)$ and $(l_j, l_k)$ are totally real. Here $\chi_0,\ldots, \chi_n\in H^1(G_S(2))$ denotes the basis dual to the system of generators $\tau_0,\ldots, \tau_n$ of $G_S(2)$ chosen as in \ref{gspres} and $tr_{r_m}: H^2(G_S(2))\to \F_2$ is the trace map corresponding to the relation $r_m$.
\end{theo}

\begin{proof}
 First assume that $i\neq j$. We set $k_1=\Q(\sqrt{l_i}), k_2=\Q(\sqrt{l_j})$. By \ref{normproposition}, it follows that there exists a prime $\mf{p}_j$ in $k_1$ over $l_j$ which is unramified in $k_1(\sqrt{\alpha})$. We have the presentation
\[
 G(K_{i,j}|\Q)=\langle s,t\ |\ s^2=t^2=(st)^4 = 1\rangle
\]
where $s,t$ are chosen as in \ref{discriminantprop}(i). Since by \ref{discriminantprop}(ii) $l_j$ ramifies in $K_{i,j}$, it follows that there is a prime $\mf{P}_j$ in $K_{i,j}$ above $l_j$ such that the inertia group $T_{\mf{P}_j}$ is generated by $s$. It follows from symmetry that there exists a prime $\mf{P}_i$ in $K_{i,j}$ above $l_i$ such that the inertia group $T_{\mf{P}_i}$ is generated by $t$. Assuming that the pair $(l_i,l_j)$ is totally real, we have $K_{i,j}\subseteq \Q_S(2)$ and we may choose primes $\mf{l}_i,\mf{l_j}$ in $\Q_S(2)$ lying above $\mf{P}_i, \mf{P}_j$. For all $0\le k\le n,\ k\neq i,j$, we choose an arbitrary prime in $\Q_S(2)$ above $l_k$. 
\vspace{10pt}

As in \ref{gspres} we have a minimal presentation $1\to R\to F\to G_S(2)\to 1$ where $F$ is the free pro-$2$-group on $x_0,\ldots, x_n$ and $x_k$ maps to $\tau_k\in T_{\mf{l}_k}$ for all $0\le k\le n$ and hence we have a projection
\[
\begin{tikzpicture}[description/.style={fill=white,inner sep=2pt}, bij/.style={below,sloped,inner sep=1.5pt}]
\matrix (m) [matrix of math nodes, row sep=1.5em,
column sep=2.5em, text height=1.5ex, text depth=0.25ex]
{\pi: F & G_S(2) & G(K_{i,j}|\Q).\\};
\path[->>,font=\scriptsize]
(m-1-1) edge node[auto] {} (m-1-2)
(m-1-2) edge node[auto] {} (m-1-3);
\end{tikzpicture}
\]
Since by \ref{discriminantprop} $K_{l_i,l_j}$ is unramified outside $\{l_i,l_j\}$, we have $\pi(x_k)=1,\ k\neq i,j$. Furthermore, $\pi(x_i)=t$, since $\pi(x_i)$ is contained in $T_{\mf{P}_i}$ and must be non-trivial. In fact, if $1\le i\le n$, this follows immediately from the observation that $\tau_i$ is a generator of the inertia subgroup $T_{\mf{l}_i}\subseteq G_S(2)$ of $\mf{l}_i$. For $i=0$ we can argue as follows: Suppose $\pi(x_0)=1$, then the restriction of $\tau_0$ to $k_1=\Q(\sqrt{2})$ is trivial. On the other hand, the restriction of $\tau_0$ to $\Q_S(2)^{ab}$ equals $(\hat{g}_0,\Q_S(2)^{ab}|\Q)$ where $\hat{g}_0$ denotes the idèle whose $2$-component is $5$ and whose other components are $1$. Since $(\hat{g}_0,\Q(\sqrt{2})|\Q)$ is the non-trivial element in $\Gal(\Q(\sqrt{2})|\Q)$, this yields a contradiction. By symmetry, we conclude that
\[
 \pi(x_k)=\left\{ \begin{array}{ll}
t,& \mbox{if}\ k=i,\\
s,& \mbox{if}\ k=j,\\
1,& \mbox{if}\ k\neq i,j.\\
\end{array}\right.
\]
Now let $k\in \{1,\ldots, n\}$ such that $k\neq i,j$ and let $\mf{p}_k$ denote a prime ideal in $k_1$ above $l_k$ unramified in $k_1(\sqrt{\alpha})$. Choose a prime $\mf{P}_k$ in $k_1k_2=\Q(\sqrt{l_i},\sqrt{l_j})$ above $\mf{p}_k$. By our assumptions on the Legendre symbols, $l_k$ is completely decomposed in $k_1 k_2$ and hence $\mf{p}_k$ splits in $k_1(\sqrt{\alpha})$ if and only if $\mf{P}_k$ splits in $k_1k_2(\sqrt{\alpha})=K_{i,j}$, i.e.\ we have
\[
  \left(\frac{\alpha|k_1}{\mathfrak{p}_k}\right) =  \left(\frac{\alpha|k_1k_2}{\mathfrak{P}_k}\right).
\]
Noting that $G(K_{l_i,l_j}|k_1k_2)$ is generated by $(st)^2$ and since by definition the element $y_k\in F$ (chosen with respect to a prolongation $\mf{l}_k$ of $\mf{P}_k$ to $\Q_S(2)$) is mapped via $\pi$ to the Frobenius of the prime lying under $\mf{l}_k$ in $K_{i,j}$, it follows that
 \[
 \pi(y_k)=\left\{ \begin{array}{ll}
(st)^2,& \mbox{if}\ [l_i,l_j,l_k]=-1,\\
1,& \mbox{if}\ [l_i,l_j,l_k]=1.\\
\end{array}\right.
\]
By \ref{discriminantprop}(i) the kernel $\tilde{R}$ of $\pi: F\to G(K_{\tilde{l}_i,\tilde{l}_j}|\Q)$ is generated by the elements $x_i^2, x_j^2,$ $(x_j x_i)^4$ and $x_l,\ l\neq i,j$ as closed normal subgroup of $F$. A straightforward computation of the Magnus expansions (cf.\ \ref{Magnusdefi}) of these elements yields
\begin{eqnarray*}
 \psi(x_i^2) & = & 1+X_i^2,\\
 \psi(x_j^2) & = & 1+X_j^2,\\
 \psi((x_j x_i)^4) & \equiv & 1\mod \deg\ge 4\\
 \psi(x_l) & = & 1+X_l.
\end{eqnarray*}
Therefore the maps $\varepsilon_{(i),2}, \varepsilon_{(j),2}, \varepsilon_{(i,j),2}$ vanish identically on $\tilde{R}$. If $[\tilde{l}_i,\tilde{l}_j,l_k]=1$, we have $y\in\tilde{R}$ and consequently $\varepsilon_{(i,j),2}(y_k) = 0$. If $[\tilde{l}_i,\tilde{l}_j,l_k]=-1$, then $\pi(y_k)=(st)^2$, i.e.\ $y_k=(x_jx_i)^2 r$ for some $r\in \tilde{R}$. This yields
\[
 \varepsilon_{(i,j),2} (y_k) =  \varepsilon_{(i,j),2}((x_ix_j)^2) + \varepsilon_{(i,j),2} (r) + \varepsilon_{(i),2}((x_ix_j)^2) \varepsilon_{(j),2}(r) = 1
\]
where we have used the product formula \cite{DV2}, Prop.1.1.22.
\vspace{10pt}

Next we consider the case $k=j$, so in particular $j\neq 0$. By definition of the Rédei symbol, if $[l_i, l_j, l_j]=1$, we have the decomposition
\[
 l_j\mc{O}_{k_1(\sqrt{\alpha_2})} = \mf{q}_1\mf{q}_2 \mf{q}_3^2
\]
with pairwise different prime ideals $\mf{q}_i$. By choice of the prime $\mf{P}_j$ of $K_{i,j}$, we have $\mf{P}_j\mid \mf{q}_1$ or $\mf{P}_j\mid \mf{q}_2$. The Frobenius automorphism of $\mf{l}_j$ maps to the trivial element of $\Gal(k_1(\sqrt{\alpha_2})|\Q)$, i.e.\ $\pi(y_j)=1$ or $\pi(y_j)=s$. We recall that the restriction of the image $\sigma_j$ of $y_j$ in $G_S(2)$ to $\Q_S(2)^{ab}$ is given by $(\hat{l}_j, \Q_S(2)^{ab}|\Q)$ where $\hat{l}_j$ denotes the idèle whose $l_j$-component equals $l_j$ and all other components are $1$, i.e.\ by class field theory the restriction of $\sigma_j$ to $k_2=\Q(\sqrt{l_j})$ is trivial. Since $s$ maps $\sqrt{l_j}$ to $-\sqrt{l_j}$, the case $\pi(y_j)=s$ cannot occur and in particular $\varepsilon_{(i,j),2}(y_k)=0$ holds. If $[\tilde{l}_i, l_j, l_j]=-1$, then $l_j$ decomposes as
 \[
 l_j\mc{O}_{k_1(\sqrt{\alpha_2})} = \mf{q}_1\mf{q_2}^3
\]
where $\mf{P}_j\mid{q}_1$. In this case the Frobenius automorphism of $\mf{l}_j$ maps to the non-trivial automorphism of $k_1(\sqrt{\alpha_2})|k_1$ and we have $\pi(y_j)= (st)^2$ or $\pi(y_j)=s(st)^2$ where again the latter case cannot occur, since $s(st)^2$ maps $\sqrt{l_j}$ to $-\sqrt{l_j}$. As in the case $k\neq j$ we conclude that $\varepsilon_{(i,j),2}(y_j)=1$. By symmetry, $\varepsilon_{(i,j),2} = 1$ if and only if $[l_i,l_j,l_i]=-1$. 
 \vspace{10pt} 

In order to determine $\varepsilon_{(i,i),2}(y_k)$ for $\ 1\le i,k\le n,\ i\neq j$, first note that by \ref{discriminantprop} $K_{i,i}|\Q$ is a cyclic extension of degree $4$ and that by \ref{totallyrealproposition}(i) $K_{l_i,l_i}\subseteq \Q_S(2)$. More precisely, $K_{i,i}$ is unramified outside $l_i$ and totally ramified at $l_i$. We can choose a minimal presentation $1\to R\to F\to G_S(2)\to 1$, such that the induced projection $\pi: F\to G(K_{l_i,l_i}|\Q)$ maps $x_l$ to $1$ for $l\neq i$ and $x_i$ to $s$ where $s$ is a generator of  $G(K_{i,i}|\Q)$. By definition it holds that
 \[
 \pi(y_k)=\left\{ \begin{array}{ll}
s^2,& \mbox{if}\ [l_i,l_i,l_k]=-1,\\
1,& \mbox{if}\ [l_i,l_i,l_k]=1.\\
\end{array}\right.
\]
Again let $\tilde{R}$ be the kernel of $\xyalign \xymatrix@C=12pt{\pi: F\ar@{->>}[r]& G(K_{\tilde{l}_i,\tilde{l}_j}|\Q)}$, i.e.\ $\tilde{R}$ is the closed normal subgroup of $F$ generated by $x_i^4, x_l,\ k\neq i$. Since
\[
 \psi(x_i^4) = 1+X_i^4,
\]
we see that in particular $\varepsilon_{(i,i),2}, \varepsilon_{i,2}$ vanish on $\tilde{R}$. Hence if $\pi(y_k)=1$, we have $\varepsilon_{(i,i),2}(y_k)=0$. If $\pi(y_k)=s^2$, then $y_k = x_i^2 r$ for some $r\in\tilde{R}$ which implies that
\[
\varepsilon_{(i,i),2} (y_k) =  \varepsilon_{(i,i),2}(x_i^2) + \varepsilon_{(i,i),2} (r) + \varepsilon_{(i),2}(x_i^2) \varepsilon_{(i),2}(r) = 1.
\]
Since $\pi(y_i)=1$, by the same argument we obtain $\mu_2(i,i,i)=0$, showing the first part of (ii).
\vspace{10pt}

Finally we consider the case $(i,j)=(0,0)$. Let $K=\Q(\zeta_{16})^+$ denote the maximal real subfield of the cyclotomic field $\Q(\zeta_{16})$. The extension $K|\Q$ is cyclic of degree $4$ over $\Q$ and $K\subseteq \Q_S(2)$. Let $s\in G(K|\Q)$ denote a generator. By definition of the symbol $[l_0,l_0,l_k]$, we have a projection 
\[
\begin{tikzpicture}[description/.style={fill=white,inner sep=2pt}, bij/.style={below,sloped,inner sep=1.5pt}]
\matrix (m) [matrix of math nodes, row sep=1.5em,
column sep=2.5em, text height=1.5ex, text depth=0.25ex]
{\pi: F & G_S(2) & G(K|\Q)\\};
\path[->>,font=\scriptsize]
(m-1-1) edge node[auto] {} (m-1-2)
(m-1-2) edge node[auto] {} (m-1-3);
\end{tikzpicture}
\]
such that for all $1\le k\le n$
\[
\pi(y_k) =  \left\{ \begin{array}{ll}
s^2,& \mbox{if}\ [l_0,l_0,l_k]=-1,\\
1,& \mbox{if}\ [l_0,l_0,l_k]=1.\\
\end{array}\right.
\]
Hence as in the case $i=j, i\neq 0$ we conclude that $\varepsilon_{(0,0},2) (y_j) = 1$ if and only if $[l_0,l_0,l_j]=-1$.
\vspace{10pt}

Summing up, for $(i,j,k), 0\le i,j\le n, 1\le k\le n$ we have the identities
\[
\varepsilon_{(i,j),2}(y_k)=\left\{ \begin{array}{ll}
1,& \mbox{if}\ [l_i,l_j,l_k]=-1,\\
0,& \mbox{if}\ [l_i,l_j,l_k]=1.\\
\end{array}\right.
\]
Hence for $1\le m\le n$ it follows from \ref{epsilonmaps} that
\begin{eqnarray*}
 tr_{r_m} \langle \chi_i, \chi_j,\chi_k \rangle_3 & = & \varepsilon_{(i,j,k),2} (r_m) = \varepsilon_{(i,j,k),2} (x_m^{l_m-1} [x_m^{-1}, y_m^{-1}])\\ & = & 
\varepsilon_{(i,j,k),2} ([x_m^{-1}, y_m^{-1}])\\
& = & \varepsilon_{(i),2} (x_m^{-1}) \varepsilon_{(j,k),2} (y_m^{-1}) - \varepsilon_{(k),2} (x_m^{-1}) \varepsilon_{(i,j),2} (y_m^{-1})\\
& = & \left\{ \begin{array}{ll}
1,& \mbox{if}\ m=k, m\neq i, [l_i,l_j,l_k]=-1,\\
1,& \mbox{if}\ m=i, m\neq k, [l_i,l_j,l_k]=-1,\\
0,& \mbox{otherwise}
\end{array}\right.
\end{eqnarray*}
where we have used the identities for the maps $\varepsilon_{\cdot,2}$ given in \cite{DV}, Prop.1.1.22. This concludes the proof.
\end{proof}

\begin{rema}
 For triple Massey products of the form $\langle \chi_i,\chi_j,\chi_k \rangle_3$, $i,j,k\neq 0$, the above proof follows the ideas given in \cite{DV2} and \cite{MM2}. The main difficulty comes from the fact that $2\in S$ in our case which amounts to calculating Massey products for $\chi_0$. 
\end{rema}

Together with Theorem \ref{cohomologicalcrit} this enables us to prove the main result of this section which gives a large supply of mild pro-$2$-groups of the form $G_S(2)$ with Zassenhaus invariant 3:

\begin{theo}
\label{wildexample} 
 Let $S=\{l_0,l_1,\ldots, l_n\}$ for some $n\ge 1$ and prime numbers $l_0=2,\ l_i\equiv 9\mod 16,\ i=1,\ldots, n$, such that the Legendre symbols satisfy
\[
 \left(\frac{l_i}{l_j}\right)_2=1,\ 1\le i,j\le n,\ i\neq j.
\]
Then $G_S(2)$ is a mild pro-$2$-group with generator rank $n+1$, relation rank $n$ and Zassenhaus invariant $\z(G)=3$.
\end{theo}
\begin{proof}
  We have already seen that $\z(G_S(2))\ge 3$. As in \ref{gsmassey} let $\chi_0,\ldots, \chi_n\in H^1(G_S(2))$ denote the dual basis of the system of generators $\tau_0,\ldots, \tau_n$. Since by assumption $l_1\equiv \ldots \equiv l_n\equiv 9\mod 16$, by \ref{gsmassey} we obtain
 \[
\tr_{r_m}\langle \chi_{0}, \chi_{0}, \chi_{k} \rangle_3 = \tr_{r_m}\langle \chi_{k}, \chi_{0}, \chi_{0} \rangle_3 = \delta_{mk} = 
 \left\{ \begin{array}{ll}
1,& \mbox{if}\ m=k,\\ 
0,& \mbox{if}\ m\neq k
\end{array}\right.
\]
for $m=1,\ldots, n$. Note that we have used the shuffle identity \cite{JGMasseyMild}, Prop.4.8. In particular, the triple Massey product is non-zero and therefore $\z(G_S(2))=3$. We apply \ref{cohomologicalcrit} with respect to the subspaces $U,V$ spanned by $\{\chi_1,\ldots, \chi_n\}$ and $\{\chi_0\}$ respectively and $e=1$. Noting that $\tr_{r_1},\ldots, \tr_{r_n}$ is a basis of $H^2(G_S(2))^\vee$, the above observation implies that the $\F_2$-linear map $U\otimes V\otimes V\longrightarrow H^2(G_S(2))$ is an isomorphism, i.e. condition (b) of \ref{cohomologicalcrit} is satisfied. Furthermore, again by \cite{JGMasseyMild}, Prop.4.8 we have $\langle \chi_0, \chi_0, \chi_0 \rangle_3 = 0$. Hence, condition (a) also holds and we conclude that $G_S(2)$ is mild with respect to the Zassenhaus filtration.
\end{proof}

In the above proof we only had to calculate Massey products of the form $\langle \chi_0,\chi_0,\chi_k\rangle_3$. In the following examples all pairs of primes $(l_i,l_j)$ are totally real, so we get an entire description of the triple Massey product:

\pagebreak

\begin{exam}\quad
\begin{itemize}
\item[(i)]
Let $S=\{l_0,\ldots, l_4\}$ where 
\[
l_0=2,\ l_1=313,\ l_2=457,\ l_3=521.
\]
We have $\left(\frac{313}{457}\right)_2=\left(\frac{313}{521}\right)_2=\left(\frac{457}{521}\right)_2=1$ and a calculation of solutions of the diophantine equation in the proof of \ref{normproposition} shows that all pairs $(l_i,l_j)$ are totally real. Using the computational algebra system MAGMA \cite{MA}, we find that the symbol $[l_i,l_j,l_k]$ is $-1$ for all permutations of the triples
\begin{eqnarray*}
 (i,j,k) &=& (1,1,3),\ (1,2,3),\ (1,3,3),\ (0,0,1),\ (0,0,2),\ (0,0,3), \\
&& (0,1,1),\ (0,2,2),\ (0,3,3),\ (0,2,3),\ (0,3,2)
\end{eqnarray*}
and $[l_i,l_j,l_k]=1$ in all other cases. By \ref{wildexample}, $G_S(2)$ is mild.
\item[(ii)] Let $S=\{l_0,l_1, l_2\}$ where 
\[
l_0=2,\ l_1=113,\ l_2=593.
\]
Computations using MAGMA \cite{MA} show that all pairs of primes in $S$ are totally real and that all Rédei symbols are equal to $1$. Hence, the triple Massey product of $G_S(2)$ is identically zero, i.e.\ we have $\z(G_S(2))\ge 4$.
\end{itemize}
\end{exam}

\section{Fabulous pro-$2$-groups with trivial cup-product}

Let $k$ be a number field and $S$ be a finite set of primes of $k$ disjoint from the places $S_p$ lying above $p$. Then the maximal pro-$p$-extension $k_S(p)|k$ unramified outside $S$ does not contain any $\Z_p$-extension and therefore its Galois group $G_S(p)$ has finite abelianization. More precisely, $G_S(p)$ is a fab pro-$p$-group:

\begin{defi}
 A pro-$p$-group $G$ is called \f{fab} if for every open subgroup $H\subseteq G$ the abelianization $H^{ab}=H/[H,H]$ is finite. We call $G$ \f{fabulous} if it is mild and fab.
\end{defi}

If a fab pro-$p$-group $G$ is of cohomological dimension $2$, it is a duality group of strict cohomological dimension 3. In particular, this holds for any fabulous group. The first examples of fabulous groups of the form $G_S(p)$ have been constructed by J.\ Labute \cite{JL} over $k=\Q$. The results of A.\ Schmidt et.\ al.\ provide an infinite supply of examples over arbitrary number fields. One should also remark the relevance of arithmetical fab groups with regard to the Fontaine-Mazur conjecture. 

However, from a group theoretical point of view, fabulous pro-$p$-groups are not yet well understood. To the author's knowledge, to date there is no known example of a non-analytic fab pro-$p$-group being explicitly described in terms of generators and relations. Therefore, it seems desirable to obtain examples of fabulous pro-$p$-groups with a more explicit knowledge of the relation structure, e.g.\ by a complete determination of the triple Massey product.

The groups $G_S(2)$ studied in the previous section are not fab since we had to allow wild ramification, i.e.\ $2\in S$. It turns out that one can produce fab (and even fabulous) quotients of these groups by adding further arithmetic conditions:

If $k$ is a number field and $S,T$ are disjoint finite sets of primes of $k$, we denote by $G_S^T(p)$ the Galois group of the maximal $p$-extension $k_S^T(p)$ of $k$ which is unramified outside $S$ and completely decomposed at the primes above $T$. 

For $k=\Q,\ p=2,\ \#T=1$ we obtain the following

\begin{theo}
\label{GSTpres}
Let $S=\{l_0,\ldots, l_n\}$ where $l_0=2$ and $l_1,\ldots, l_n$ are prime numbers $\equiv 1\modulo 8$ satisfying $\left(\frac{l_i}{l_j}\right)_2=1,\ 1\le i,j\le n,\ i\neq j$. Furthermore, let $T=\{q\}$ where $q\not\in S$ is a prime number $\equiv 5\mod 8$, such that the following conditions are satisfied:
\begin{itemize}
 \item[(1)] $\left(\frac{q}{l_i}\right)=1$ for all $i=1,\ldots, n-1$,
 \item[(2)] $\left(\frac{q}{l_n}\right)=-1$.
\end{itemize}
Then for the pro-$2$-group $G_S^T(2)$ the following holds: 
\begin{itemize}
 \item[\rm (i)] $G_S^T(2)$ has generator rank $h^1(G_S^T(2))=n$, relation rank $h^2(G_S^T(2)) \le n$ and Zassenhaus invariant $\z(G_S^T(2))\ge 3$.
 \item[\rm (ii)] Assume that the pair $(l_i,l_j)$ is totally real for all $0\le i,j\le n,\ i\neq j$. Then $G_S^T(2)$ possesses a presentation $G_S^T(2)=\langle \ol{x}_1,\ldots, \ol{x}_{n}\ |\ \ol{r}_1,\ldots, \ol{r}_{n}\rangle$ such that the triple Massey product
\[
\langle\cdot, \cdot, \cdot\rangle_3: H^1(G_S^T(2))\times H^1(G_S^T(2)) \times H^1(G_S^T(2)) \longrightarrow H^2(G_S^T(2))
\]
is given by
\[
(-1)^{tr_{\ol{r}_m}\langle \ol{\chi}_{i}, \ol{\chi}_{j}, \ol{\chi}_{k} \rangle_3} =
 \left\{ \begin{array}{ll}
\lbracket l_i,l_j,l_k \rbracket,& \mbox{if}\ m=k,\ m\neq i,\ i,j\neq n,\\
\lbracket l_i,l_j,l_k \rbracket\cdot\lbracket l_0,l_j,l_k\rbracket,& \mbox{if}\ m=k,\ i=n,j\neq n,\\
\lbracket l_i,l_j,l_k\rbracket\cdot\lbracket l_0,l_j,l_k\rbracket,& \mbox{if}\ m=k,\ m\neq i,\ i\neq n=j,\\
\lbracket l_i,l_j,l_k\rbracket\cdot\lbracket l_0,l_j,l_k\rbracket,& \mbox{if}\ m=k,\ i=j=n,\\
\lbracket l_j,l_k,l_i\rbracket,& \mbox{if}\ m=k,\ m\neq i, i,j\neq n,\\
\lbracket l_j,l_k,l_i\rbracket\cdot\lbracket l_0,l_j,l_k\rbracket,& \mbox{if}\ m=i,\ k=n,\ j\neq n,\\
\lbracket l_j,l_k,l_i\rbracket\cdot\lbracket l_0,l_k,l_i\rbracket,& \mbox{if}\ m=i,\ m\neq k,\ k\neq n=j,\\
\lbracket l_j,l_k,l_i\rbracket\cdot\lbracket l_0,l_0,l_i\rbracket,& \mbox{if}\ m=i,\ k=j=n,\\ 
1,& \mbox{otherwise}
\end{array}\right.
\]
for $m=1,\ldots, n-1$ and
\[
(-1)^{tr_{\ol{r}_n}\langle \ol{\chi}_{i}, \ol{\chi}_{j}, \ol{\chi}_{k} \rangle_3} =
 \left\{ \begin{array}{ll}
\lbracket l_i,l_j,l_n \rbracket,& \mbox{if}\ k=n,\ i,j\neq n,\\
\lbracket l_i,l_j,l_n \rbracket\cdot\lbracket l_i,l_0,l_n\rbracket,& \mbox{if}\ k=j=n,\ i\neq n,\\
\lbracket l_j,l_k,l_n \rbracket,& \mbox{if}\ i=n,\ k,j\neq n,\\
\lbracket l_j,l_k,l_n \rbracket\cdot\lbracket l_0,l_k,l_n\rbracket,& \mbox{if}\ i=j=n,\ k\neq n\\
1,& \mbox{otherwise}
\end{array}\right.
\]
where $\{\ol{\chi}_1,\ldots, \ol{\chi}_{n}\}$ denotes the basis of $H^1(G_S^T(2))$ dual to $\ol{x}_1,\ldots, \ol{x}_{n}$.
 \end{itemize}
Assuming in addition that the Leopoldt conjecture holds for all number fields $k$ contained in $\Q_S^T(2)$ and the prime $2$, we have:
\begin{itemize}
 \item[\rm (iii)] $G_S^T(2)$ is a fab pro-$2$-group and $h^2(G_S^T(2))=n$.
\end{itemize}
\end{theo}
\begin{proof}
Keeping the notation and choices of the elements $\tau_0,\ldots, \tau_n\in G_S(2)$ as in the previous sections, by \ref{gspres} the group $G_S(2)$ has a minimal presentation
\[
\begin{tikzpicture}[description/.style={fill=white,inner sep=2pt}, bij/.style={below,sloped,inner sep=1.5pt}]
\matrix (m) [matrix of math nodes, row sep=3em,
column sep=2.5em, text height=1.5ex, text depth=0.25ex]
{1 & R & F & G_S(2) & 1\\};
\path[->,font=\scriptsize]
(m-1-1) edge node[auto] {} (m-1-2)
(m-1-2) edge node[auto] {} (m-1-3)
(m-1-3) edge node[auto] {$\pi$} (m-1-4)
(m-1-4) edge node[auto] {} (m-1-5);
\end{tikzpicture}
\]
where $F$ is the free pro-$p$-group on $x_0,\ldots, x_n$, $\pi$ maps $x_i$ to $\tau_i$ and $R$ is generated by $r_i = x_i^{l_i-1}[x_i^{-1}, y_i^{-1}],\ i=1, \ldots, n$ as closed normal subgroup of $F$. We fix a prime $\mf{Q}$ of $\Q_S(2)$ lying above $q$ and denote by $G_\mf{Q}\subseteq G_S(2)$ the decomposition group of $\mf{Q}$, which is generated as closed subgroup by the Frobenius automorphism $\sigma_\mf{Q}$ of $\mf{Q}$. Furthermore, we choose an arbitrary lift $r_\mf{Q}\in \pi^{-1}(\sigma_\mf{Q})$. We obtain a commutative exact diagram
\begin{eqnarray*}
 \label{nonminimaldiagram}
\begin{tikzpicture}[description/.style={fill=white,inner sep=2pt}, bij/.style={below,sloped,inner sep=1.5pt}]
\matrix (m) [matrix of math nodes, row sep=1.5em,
column sep=2.5em, text height=1.5ex, text depth=0.25ex]
{  & &  & 1  &\\
   & &  & (G_\mf{Q}) &\\
 1 & R & F & G_S(2) & 1\\
 1 & \tilde{R} & F & G_S^T(2) & 1\\
 &  &  & 1\\      
};
\path[->,font=\scriptsize]
(m-3-1) edge node[auto] {} (m-3-2)
(m-3-2) edge node[auto] {} (m-3-3)
(m-3-3) edge node[auto] {$\pi$} (m-3-4)
(m-3-4) edge node[auto] {} (m-3-5)
(m-4-1) edge node[auto] {} (m-4-2)
(m-4-2) edge node[auto] {} (m-4-3)
(m-4-3) edge node[auto] {$\tilde{\pi}$} (m-4-4)
(m-4-4) edge node[auto] {} (m-4-5)
(m-1-4) edge node[auto] {} (m-2-4)
(m-2-4) edge node[auto] {} (m-3-4)
(m-3-4) edge node[auto] {} (m-4-4)
(m-4-4) edge node[auto] {} (m-5-4);
\path[-,font=\scriptsize]
(m-3-3) edge[double,double distance=1.5pt] node[auto] {} (m-4-3);
\path[right hook->,font=\scriptsize]
(m-3-2) edge node[auto] {} (m-4-2);
\end{tikzpicture}
\end{eqnarray*}
where $\tilde{R}$ denotes the closed normal subgroup of $F$ generated by $r_1,\ldots, r_n, r_\mf{Q}$ and $(G_\mf{Q})\subseteq G_S(2)$ is the closed normal subgroup generated by $G_\mf{Q}$. Let $\hat{q}\in I=I_\Q$ denote the idèle of $\Q$ whose $q$-th component equals $q$ and all other components are $1$. Then the restriction of $\sigma_{\mf{Q}}$ to $\Q_S[2]$, the maximal abelian subextension of exponent $2$ in $\Q_S(2)|\Q$, is given by $(\hat{q}, \Q_S[2]|\Q)$ where $(\cdot, \Q_S[2]|\Q)$ denotes the global norm residue symbol. By choice of $q$, we have the idèlic congruence
\begin{eqnarray*}
 \hat{q} && =  (1,\ldots, 1,q,1,\ldots)\\
	&& \equiv (\frac{1}{q},\ldots,\frac{1}{q},1,\frac{1}{q},\ldots)\\
	&& \equiv  \hat{g}_0 \hat{g}_n \mod I_S I^2 \Q^\times
\end{eqnarray*}
where 
\[
 I_S= \prod_{l\in S} \{1\} \times \prod_{l\not\in S} U_l
\]
and $\hat{g}_i$ are the idèles as constructed in the definition of the generators $\tau_i$ of $G_S(2)$ (cf.\ \ref{gspres}). Since by class field theory $\Gal(\Q_S[2]|\Q) = G_S(2) / (G_S(2))_{(2)} \cong I/I_S I^2 \Q^\times$, it follows that
\[
 \sigma_\mf{Q}\equiv \tau_0 \tau_n \mod (G_S(2))_{(2)},\quad r_\mf{Q}\equiv x_0 x_n\mod F_{(2)}.
\]
In particular, we see that $r_\mf{Q}\not\in F_{(2)}$ and therefore the presentation of $G_S^T(2)$ given by the bottom horizontal line in the above diagram is \it not \rm minimal. In order to obtain a minimal presentation, let $\ol{F}$ be the free pro-$p$-group on $n$ generators $\ol{x}_1,\ldots, \ol{x}_n$. Noting that $r_\mf{Q}, x_1,\ldots, x_n$ is a basis of $F$, the mapping $r_\mf{Q}\longmapsto 1,\ x_i\longmapsto \ol{x}_i,\ i=1,\ldots, n$ yields a well-defined surjective homomorphism $\psi: F\longrightarrow \ol F$. Let $s$ be the section of $\psi$ mapping $\ol{x}_i$ to $x_i,\ i=1,\ldots, n$ and set $\ol{R}=\psi(\tilde{R})$. We obtain the commutative exact diagram
\begin{eqnarray*}
\label{minimaldiagram}
 \begin{tikzpicture}[description/.style={fill=white,inner sep=2pt}, bij/.style={below,sloped,inner sep=1.5pt}]
\matrix (m) [matrix of math nodes, row sep=2.5em,
column sep=2.5em, text height=1.5ex, text depth=0.25ex]
{1 & \tilde{R} & F & G_S^T(2) & 1\\
 1 & \ol{R} & \ol{F} \\      
};
\path[->,font=\scriptsize]
(m-1-1) edge node[auto] {} (m-1-2)
(m-1-2) edge node[auto] {} (m-1-3)
(m-1-3) edge node[auto] {$\tilde{\pi}$} (m-1-4)
(m-1-4) edge node[auto] {} (m-1-5)
(m-2-1) edge node[auto] {} (m-2-2)
(m-2-2) edge node[auto] {} (m-2-3)
(m-2-3) edge[bend right=40] node[auto,swap] {$s$} (m-1-3);
\path[->>,font=\scriptsize]
(m-1-2) edge node[auto] {} (m-2-2)
(m-1-3) edge node[auto,swap] {$\psi$} (m-2-3);
\path[->>,font=\scriptsize]
(m-2-3) edge[dotted] node[auto] {} (m-1-4);
\end{tikzpicture}
\end{eqnarray*}
where the composition $\tilde{\pi}\circ s$ is surjective, since $x_0\equiv x_n \mod \tilde{R} F_{(2)}$ and therefore $G_S^T(2)$ is generated by the images of $\tau_1,\ldots, \tau_n$. Clearly,  $\ol{R}$ is generated by $\ol{r}_i=\psi(r_i),\ i=1,\ldots, n$. By the assumptions made for the primes in $S$, we have $r_i\in F_{(3)}$ and therefore also $\ol{R}\subseteq \ol{F}_{(3)}$. In particular, we obtain the \it minimal \rm presentation
\[
 G_S^T(2) = \ol{F}/\ol{R} = \langle \ol{x}_1,\ldots, \ol{x}_n\ |\  \ol{r}_1,\ldots, \ol{r}_n\rangle
\]
for $G_S^T(2)$. This yields $h^1(G_S^T(2)) = n,\ h^2(G_S^T(2))\le n$ and the cup-product $H^1(G_S^T(2))\times H^1(G_S^T(2))\stackrel{\cup}{\to} H^2(G_S^T(2))$ is trivial, i.e.\ we have proven (i).

Next we calculate the triple Massey product of $G_S^T(2)$. Let $\chi_0,\ldots, \chi_n\in H^1(G_S(2)) = H^1(F)$ and $\ol{\chi}_1,\ldots, \ol{\chi}_n$ denote the bases dual to $x_0,\ldots, x_n$ and $\ol{x}_1,\ldots, \ol{x}_n$ respectively. Since $\psi(x_0)\equiv \psi(x_n)=\ol{x}_n\mod F_{(2)}$, the inflation map $\infl: H^1(G_S^T(2))\longrightarrow H^1(G_S(2))$ is given by
\[
\ol{\chi}_i\longmapsto \chi_i,\ i=1,\ldots, n-1,\quad \ol{\chi}_n\longmapsto \chi_0+\chi_n.
\]
Furthermore, we have the surjective homomorphism
 \[
\begin{tikzpicture}[description/.style={fill=white,inner sep=2pt}, bij/.style={above,sloped,inner sep=1.5pt}, column 1/.style={anchor=base east},column 2/.style={anchor=base west}]
\matrix (m) [matrix of math nodes, row sep=0em,
column sep=2.5em, text height=1.5ex, text depth=0.25ex]
{ \inflv:\ H^2(G_S(2))^\vee & H^2(G_S^T(2))^\vee,\\
  \tr_{r_i} & \tr_{\ol{r}_i},\ i=1,\ldots, n.\\};
\path[->>,font=\scriptsize]
(m-1-1) edge node[auto] {} (m-1-2);
\path[|->,font=\scriptsize]
(m-2-1) edge node[auto] {} (m-2-2);
\end{tikzpicture}
\]
Since Massey products commute with the inflation maps, for any $1\le i,j,k,m\le n$ we have
\begin{eqnarray*}
 \tr_{\ol{r}_m} \langle \ol{\chi}_i,\ol{\chi}_j,\ol{\chi}_k\rangle_3 & = & \inflv (\tr_{r_m}) \langle \ol{\chi}_i,\ol{\chi}_j,\ol{\chi}_k\rangle_3\\
& = & \tr_{r_m} \infl \langle \ol{\chi}_i,\ol{\chi}_j,\ol{\chi}_k\rangle_3\\
& = & \tr_{r_m} \langle \infl \ol{\chi}_i,\infl \ol{\chi}_j,\infl \ol{\chi}_k\rangle_3.
\end{eqnarray*}
We can now deduce (ii) by a direct calculation using \ref{gsmassey} and the shuffle property of the triple Massey product.

Let $H\subseteq G_S^T(2)$ be an open subgroup and $k\subseteq \Q_S^T(2)$ the corresponding fixed field. Assume that $H^{ab}$ is infinite. Since it is finitely generated, it has a quotient isomorphic to $\Z_2$. Assuming that the Leopoldt conjecture holds for $k$ and $2$, this must be the cyclotomic $\Z_2$-extension, since $k$ is totally real (e.g.\ see \cite{NSW}, Th.10.3.6). However, the primes above $q$ cannot be completely decomposed in the cyclotomic $\Z_2$-extension of $k$ which yields a contradiction. Hence, $H^{ab}$ is finite showing that $G_S^T(2)$ is a fab group. In particular, $(G_S^T(2))^{ab}$ is finite. Consequently $h^2(G_S^T(2))\ge h^1(G_S^T(2))$ and hence equality holds.
\end{proof}

\begin{exam}
\label{fabexam}
 The sets $S=\{2, 17, 7489, 15809\}, T=\{5\}$ satisfy the assumptions made in \ref{GSTpres}. Choosing $\ol{x}_i,\ol{r}_i,\ol{\chi}_i$ as in \ref{GSTpres}, computations of the Rédei symbols with MAGMA show that
\begin{eqnarray*}
 \tr_{\ol{r}_1}\langle \ol{\chi}_{i}, \ol{\chi}_{j}, \ol{\chi}_{k} \rangle_3 \neq 0\ \Leftrightarrow& (i,j,k)\in& \{(1,1,3), (1,2,3), (1,3,2), (1,3,3),\\ &&\phantom{\{}  (2,3,1), (3,1,1), (3,2,1), (3,3,1)\},\\
 \tr_{\ol{r}_2}\langle \ol{\chi}_{i}, \ol{\chi}_{j}, \ol{\chi}_{k} \rangle_3 \neq 0\ \Leftrightarrow& (i,j,k)\in& \{(1,3,2),(2,1,3),(2,3,1),(3,1,2)\},\\
 \tr_{\ol{r}_3}\langle \ol{\chi}_{i}, \ol{\chi}_{j}, \ol{\chi}_{k} \rangle_3 \neq 0\ \Leftrightarrow& (i,j,k)\in& \{(1,1,3), (1,2,3), (2,1,3), (3,1,1),\\ &&\phantom{\{} (3,1,2), (3,2,1)\}.
\end{eqnarray*}
Hence, expressing the defining relations $\ol{r}_1,\ol{r}_2, \ol{r}_3\in R$ in the minimal presentation $1\to R\to F\to G_S^T(2)\to 1$ in terms of \f{basic commutators} of degree $3$ using \cite{DV2}, Prop.1.3.3, we obtain
\begin{eqnarray*}
\ol{r}_1 & \equiv & \ocom{1}{3}{1}\ \ocom{1}{3}{3}\ \ocom{2}{3}{1}\ \mod F_{(4)}\\
\ol{r}_2 & \equiv & \ocom{1}{3}{2}\ \mod F_{(4)},\\
\ol{r}_3 & \equiv & \ocom{1}{3}{1}\ \ocom{1}{3}{2}\ \ocom{2}{3}{1}\ \mod F_{(4)}.
\end{eqnarray*}
We claim that $G_S^T(2)$ is mild. To this end let $U=\langle \chi_1 \rangle,\ V=\langle \chi_2,\chi_3 \rangle$. By the above calculations the triple Massey product $\langle \cdot, \cdot, \cdot \rangle_3$ is trivial on $V\times V\times V$ and maps $U\times V\times V$ surjectively onto $H^2(G_S^T(2))$. Hence, the mildness of $G_S^T(2)$ follows by \ref{cohomologicalcrit}. Assuming the Leopoldt conjecture, $G_S^T(2)$ is a fabulous pro-$2$-group. 
\end{exam}

To the author's knowledge, this yields the first known example of a fabulous pro-$p$-group with trivial cup-product and also the first example of a fabulous pro-$p$-group with generator rank $\le 3$. J.\ Labute, C.\ Maire and J.\ Miná\v{c} have announced analogous results for odd $p$.
\bibliographystyle{plain}
\bibliography{promotion}

\vspace{20pt}
\address{\noindent Mathematisches Institut\\ Universität Heidelberg\\ Im Neuenheimer Feld 288\\ 69120 Heidelberg\\ Germany\\ \phantom{e-mail} \\e-mail: gaertner@mathi.uni-heidelberg.de}

\end{document}